\documentclass[10pt,twocolumn,twoside]{IEEEtran}
\title{\bf \huge Optimal Placement of Virtual Inertia in Power Grids}
\author{Bala Kameshwar Poolla, Saverio Bolognani, Florian D\"orfler%
\thanks{This material is supported by ETH start-up funds and the SNF Assistant Professor Energy Grant \#160573.
B.K. Poolla, S. Bolognani, and F. D\"orfler are with the Automatic Control Laboratory at the Swiss Federal Institute of Technology (ETH) Z\"urich, Switzerland.
Emails: {\tt \{bpoolla,bsaverio,dorfler\}@ethz.ch}.} }

\usepackage[utf8]{inputenc}
\usepackage[T1]{fontenc}
\usepackage{enumitem}
\usepackage[margin=1.5cm]{geometry}

\usepackage{graphicx,xcolor}
\usepackage[vlined,ruled]{algorithm2e}
\usepackage[font={small},labelfont=it]{caption} 
\graphicspath{{./fig/}}	

\usepackage{cite}     
\usepackage{dsfont} 	

\usepackage[cmex10]{amsmath}
\usepackage{amsthm}
\usepackage{amsfonts}
\usepackage{amssymb}
\usepackage{mathrsfs}

\newtheoremstyle{bfnote}%
{}{}%
{\itshape}{}%
{\bfseries}{.}%
{ }%
{\thmname{#1}\thmnumber{ #2}\thmnote{ (#3)}}

\theoremstyle{bfnote}
\definecolor{m}{RGB}{0,56,167}

\newtheorem{theorem}{Theorem}
\newtheorem{lemma}[theorem]{Lemma}
\newtheorem{corollary}[theorem]{Corollary}

\newtheorem{assumption}{Assumption}

\newcommand*\mc[0]{\mathcal}   

\newcommand{\until}[1]{\{1,\dots,#1\}}

\newcommand{\real}[0]{\mathbb R}

\DeclareSymbolFont{bbold}{U}{bbold}{m}{n}
\DeclareSymbolFontAlphabet{\mathbbold}{bbold}

\newcommand{\vectorones}[1][]{\mathds{1}_{#1}}
\newcommand{\vectorzeros}[1][]{\mathbbold{0}_{#1}}

\newcommand*\Rank[0]{\mbox{rank}}

\newcommand\diag[0]{\textup{diag}} 
\DeclareMathOperator{\trace}{Trace}

\usepackage{standalone}
\usepackage{tikz}
\usetikzlibrary{positioning, shapes}
\usetikzlibrary{decorations.pathreplacing}
\usetikzlibrary{arrows}

\usetikzlibrary{calc,trees,positioning,arrows,chains,shapes.geometric,%
    decorations.pathreplacing,decorations.pathmorphing,shapes,%
    matrix,shapes.symbols}

\tikzset{
>=stealth',
  punktchain/.style={rectangle, rounded corners, 
    draw=black, very thick,text width=10em, 
    minimum height=3em, text centered, on chain},
  line/.style={draw, thick, <-},
  element/.style={tape,top color=white,bottom color=blue!50!black!60!,
    minimum width=8em,draw=blue!40!black!90, very thick,
    text width=10em, minimum height=3.5em, text centered, on chain},
  every join/.style={->, thick,shorten >=1pt},
  tuborg/.style={decorate},
  tubnode/.style={midway, right=2pt},
}
\usepackage{pgfplots}
\pgfplotsset{compat=1.6,
        scaled x ticks = false, 
        xticklabel style={/pgf/number format/fixed,/pgf/number format/precision=3},
        } 

\newcommand\oprocendsymbol{\hbox{$\square$}}
\newcommand\oprocend{\relax\ifmmode\else\unskip\hfill\fi\oprocendsymbol}

\usepackage{dblfloatfix} 
\usepackage{subcaption}

\def\QEDopen{{\setlength{\fboxsep}{0pt}\setlength{\fboxrule}{0.2pt}\fbox{\rule[0pt]{0pt}{1.3ex}\rule[0pt]{1.3ex}{0pt}}}}
\def\QED{\QEDopen} 

\def\extraconstraint{\ \ P z_0 = \vectorzeros[2n]}

\begin{document}
\maketitle
\thispagestyle{empty}
\pagestyle{empty}

\begin{abstract}
A major transition in the operation of electric power grids is the replacement of synchronous machines by distributed generation connected via power electronic converters. The accompanying ``loss of rotational inertia'' and the fluctuations by renewable sources jeopardize the system stability, as testified by the ever-growing number of frequency incidents. As a remedy, numerous studies demonstrate how virtual inertia can be emulated through various devices, but few of them address the question of ``where'' to place this inertia. It is, however, strongly believed that the placement of virtual inertia hugely impacts system efficiency, as demonstrated by recent case studies. In this article, we carry out a comprehensive analysis in an attempt to address the optimal inertia placement problem.
We consider a linear network-reduced power system model along with an $\mathscr{H}_2$ performance metric accounting for the network coherency. The optimal inertia placement problem turns out to be non-convex, yet we provide a set of closed-form global optimality results for particular problem instances as well as a computational approach resulting in locally optimal solutions. Further, we also consider the robust inertia allocation problem, wherein the optimization is carried out accounting for the worst-case disturbance location. We illustrate our results with a three-region power grid case study and compare our locally optimal solution with different placement heuristics in terms of different performance metrics. 
\end{abstract}


\section{Introduction}

As we retire more and more synchronous machines and replace them with renewable sources interfaced with power electronic devices, the stability of the power grid is jeopardized. This has been recognized as one of the prime concerns by  transmission system operators \cite{WW-KE-GB-KJ:15,MM-BF-BK-MS:15}. Both in transmission grids as well as in microgrids, low inertia levels together with variable renewable generation lead to large frequency swings.

Not only are low levels of inertia troublesome, but particularly spatially heterogeneous and time-varying inertia profiles can lead to destabilizing effects, as shown in an interesting two-area case study \cite{AU-TB-GA:14}. It is not surprising that rotational inertia has been recognized as a key ancillary service for power system stability, and a plethora of mechanisms have been proposed for inertia emulation (also known as virtual or synthetic inertia) \cite{SN-SD-MCC:13,HB-TI-YM:14,SD-JA:13} through a variety of devices,  ranging from wind turbine control \cite{MJ-SWHDH-WLK-FJA:06} over flywheels to batteries \cite{MK-TB-AU-GA:15}. Inertia monitoring \cite{PMA-CSS-GAT-AMC-MEB:15} and markets have also been suggested \cite{EE-GV-AT-BK-MM-MO:14}. In this article, we pursue the questions raised in \cite{AU-TB-GA:14} regarding the detrimental effects of spatially heterogeneous inertia profiles, and how they can be alleviated by inertia emulation throughout the grid. In particular, we are interested in the allocation problem ``where to optimally place the inertia''\,?

The problem of inertia allocation has been hinted at before  \cite{AU-TB-GA:14}, but we are aware only of the study \cite{TSB-TL-DJH:15} explicitly addressing the problem. In \cite{TSB-TL-DJH:15}, the grid is modeled by the linearized swing equations, and eigenvalue damping ratios as well as transient overshoots (estimated from the system modes) are chosen as optimization criteria for placing virtual inertia and damping. The resulting problem is non-convex, but a sequence of approximations leads to some insightful results. 

In comparison to \cite{TSB-TL-DJH:15}, we focus on network coherency as an alternative performance metric, that is, the amplification of stochastic or impulsive disturbances via a quadratic performance index measured by the $\mathscr H_{2}$ norm \cite{KZ-JCD-KG:96}. As performance index, we choose a classic coherency criterion penalizing angular differences and frequency excursions, which has recently been popularized for consensus and synchronization studies \cite{EL-SZ:12,BB-MRJ-PM-SP:12,MF-FL-MRJ:14,MF-XZ-FL-MRJ:14,TS-IS-JL-FD:15,MS-NM:14}  as well as in power system analysis and control \cite{ES-BB-DFG:15,FD-MJ-MC-FB:13a,XW-FD-MJ:15a}. We feel that this $\mathscr{H}_{2}$ performance metric is not only more tractable than spectral metrics, but it is also very meaningful for the problem at hand: it measures the effect of stochastic fluctuations (caused by loads and/or variable renewable generation) as well as impulsive events (such as faults or deterministic frequency errors caused by markets) and quantifies their amplification by a coherency index directly related to frequency volatility. Finally, in comparison to \cite{TSB-TL-DJH:15},  the damping or droop  coefficients are not decision variables in our problem setup, since these are determined by the system physics (in case of damping), the outcome of primary reserve markets (in case of primary control), or scheduled according to cost coefficients, ratings, or grid-code requirements \cite{FD-JWSP-FB:14a}. 

The contributions of this paper are as follows. We provide a comprehensive {modeling and analysis} framework for the inertia placement problem in power grids to optimize an $\mathscr{H}_2$ coherency index subject to capacity and budget constraints.
The optimal inertia placement problem is characteristically non-convex, yet we are able to provide explicit upper and lower bounds on the performance index. Additionally, we show that the problem admits an elegant and strictly convex reformulation for a performance index reflecting the effort of primary control which is often advocated as a remedy to low-inertia stability issues. In this case, the optimal inertia placement problem reduces to a standard resource allocation problem, where the cost of each resource is proportional to the ratio of expected disturbance over inertia. 

A similar simplification of the problem is obtained under some reasonable assumptions on the ratio between the disturbance and the damping coefficient at every node.
For the case of a two-area network, a closed-form analysis is possible, and a series of observations are discussed.

Furthermore, we develop a computational approach based on a gradient formula that allows us to find a locally optimal solution for large networks and arbitrary parameters.
We show how the combinatorial problem of allocating a limited number of inertia-emulating units can be also incorporated into this numerical method via a sparsity-promoting approach. Finally, any system norm such as  $\mathscr H_{2}$ assumes that the location of the disturbance (or a distribution thereof) is known. While empirical fault distributions are usually known based on historical data, the truly problematic faults in power grids are rare events that are poorly captured by any  disturbance distribution. To safeguard against such faults, we also present a robust formulation of the inertia allocation problem in which we optimize the $\mathscr{H}_2$ norm with respect to the worst possible disturbance.

A detailed three-region network has been adopted as a case study for the presentation of the proposed method. The numerical results are also illustrated via time-domain simulations, that demonstrate how an optimization-based allocation exhibits superior performance (in different performance metrics) compared to heuristic placements, and, perhaps surprisingly, the optimal allocation also uses less effort to emulate inertia.

From the methodological point of view, this paper extends the $\mathscr H_{2}$ performance analysis of second-order consensus systems to non-uniform damping, inertia, and input matrices (disturbance location). This technical contribution is essential for the application that we are considering, as these parameters dictate the optimal inertia allocation in an intertwined way.

The remainder of this section introduces some notation. Section \ref{Section: Problem Formulation} motivates our system model and the coherency performance index. Section \ref{Section: Optimal inertia allocation} presents numerical inertia allocation algorithms for general networks and provides explicit results for certain instances of cost functions and problem scenarios. Section \ref{Section: Case study} presents a case study on a three-region network accompanied with time-domain simulations and a spectral analysis. Finally, Section \ref{Section: Conclusions} concludes the paper.

\paragraph*{Notation}

We denote the $n$-dimensional vectors of all ones and zeros by $\vectorones[n]$ and $\vectorzeros[n]$.
Given an index set $\mathscr{I}$ with cardinality $|\mathscr{I}|$ and a real-valued array $\{{x_1}, \dots, x_{|\mathscr{I}|}\}$, we denote by $x \in \real^{|\mathscr{I}|}$ the vector obtained by stacking the scalars $x_i$ and by $\text{diag}\{{x_i}\}$ the associated diagonal matrix.
The vector $\mathbbold e_{i}$ is the $i$-th vector of the canonical basis for $\real^n$.
\section{Problem Formulation}
\label{Section: Problem Formulation}

\subsection{System model}
Consider a power network modeled by a graph with nodes (buses) $\mathcal{V} =\{1, \dots, n\}$ and edges (transmission lines)  $\mathcal{E} \subseteq \mathcal{V} \times \mathcal{V}$. We consider a small-signal version of a network-reduced power system model \cite{PWS-MAP:98,PK:94}, where passive loads are eliminated via  Kron reduction \cite{FD-FB:11d}, and the network is reduced to active buses $i$ with linearized dynamics
\begin{equation}
\label{eq:basic}
m_{i} \ddot \theta_{i} + d_{i} \dot \theta_{i} = {p_{\text{in}, i} - p_{e,i}}\, , \quad i \in \{1, \dots, n \}\, ,
\end{equation}
 where ${p_{\text{in}, i}}$ and ${p_{e,i}}$ refer to the power input and electrical power output, respectively. If bus $i$ is a synchronous machine, then \eqref{eq:basic} describes the electromechanical swing dynamics for the generator rotor angle $\theta_{i}$ \cite{PWS-MAP:98,PK:94}, $m_{i}>0$ is the generator's rotational inertia, and  $d_{i}>0$ accounts for frequency damping or primary speed droop control (neglecting ramping limits). If bus~$i$ connects to a renewable or battery source interfaced with a power electronics inverter operated in grid-forming mode \cite{QCZ-TH:13,JS-DZ-RO-AS-TS-JR:15}, then $\theta_{i}$ is the voltage phase angle, $d_{i}>0$ is the droop control coefficient, and $m_{i}>0$ accounts for power measurement time constant \cite{JS-DG-JR-TS:13}, a control gain \cite{IAH-EMF:08}, or arises from virtual inertia emulation through a dedicated controlled device~\cite{SN-SD-MCC:13,HB-TI-YM:14,SD-JA:13}. Finally, the dynamics \eqref{eq:basic} may also arise from frequency-dependent or actively controlled frequency-responsive loads~\cite{PK:94}.
In general, each bus $i$ will host an ensemble of these devices, and the quantities $m_i$ and $d_i$ are lumped parametrizations of their aggregate behavior.

Under the assumptions of identical unit voltage magnitudes, purely inductive lines, and a small-signal approximation, the electrical power output at the terminals is given by 
\cite{PK:94}
\begin{equation}
\label{eq:elec}
{p_{e,i}}= \sum_{j=1}^{n} b_{ij} (\theta_i -\theta_j), \quad i \in \{1, \dots, n \}\, ,
\end{equation}
 where $b_{ij} \geq 0$ is the susceptance between nodes $\{i,j\} \in \mathcal E$.

 The state space representation of  the system (\ref{eq:basic})-(\ref{eq:elec}) is then
\begin{equation}
\label{eq:dynamics}
\left[
\setlength\arraycolsep{1.5pt}\begin{array}{cc}
\dot{\theta}\\
\dot{\omega}\\
\end{array}
\right]=
\left[
\setlength\arraycolsep{1.5pt}\begin{array}{cc}
{0} & I \\
-{M}^{-1}L & -{M}^{-1}{D}\\
\end{array}
\right]\left[
\setlength\arraycolsep{1.5pt}\begin{array}{cc}
\theta\\
\omega\\
\end{array}
\right] +
\left[
\setlength\arraycolsep{1.5pt}\begin{array}{cc}
0\\
{M}^{-1}\\
\end{array}
\right]
{{p_\text{in}}}\,,
\end{equation}
where $M={\textup{diag}}\{m_{i}\}$ and $D ={\textup{diag}}\{d_{i}\}$ are the diagonal matrices of inertia and damping/droop coefficients, and $L=L^{\mathsf{T}} \in \real^{n \times n}$ is the network Laplacian (or susceptance) matrix with off-diagonal elements $l_{ij} = -b_{ij}$ and diagonals $l_{ii} = \sum_{j=1,j\neq i}^{n} b_{ij}$. The states $(\theta,\omega) \in \real^{2n}$ are the stacked vectors of angles and frequencies and ${p_\text{in}} \in \real^{n}$ is the net power input -- all of which are deviation variables from nominal values. 

\subsection{Coherency performance metric}
We consider the linear power system model \eqref{eq:dynamics} driven by the inputs ${p_{\text{in}, i}}$ accounting either for faults or non-zero initial values (modeled as impulses) or for random fluctuations in renewables and loads.  We are interested in the energy expended in returning to the steady-state configuration, expressed as a quadratic cost of the angle differences and frequency displacements:
\begin{equation}
	 \mathlarger{\int_0^\infty} \mathlarger{ \bigg\{\sum_{i,j =1}^{n}} {a_{ij}} (\theta_{i}(t)-\theta_{j}(t))^2 + {\mathlarger \sum_{i=1}^{n}} {s_{i}}\, \omega_{i}^2(t) \, \bigg\}\, {\text{d}t}\,.
	\label{eq: energy}
\end{equation}
Here, $s_i$ are positive scalars and we assume that the nonnegative scalars {$a_{ij} = a_{ji} \geq 0$} induce a connected graph -- not necessarily identical to the power grid itself. We denote by $S$ the matrix $\diag\{ s_i\}$, and by $N$ the Laplacian matrix of the graph induced by the {$a_{ij}$}. In this compact notation, $N = L$ would be an example of local error penalization \cite{BB-MRJ-PM-SP:12,EL-SZ:12}, while $N =  I_{n} - \vectorones[n]\vectorones[n]^{\mathsf{T}}/n$ penalizes global errors.

Aside from consensus and synchronization studies \cite{EL-SZ:12,BB-MRJ-PM-SP:12,MF-FL-MRJ:14,MF-XZ-FL-MRJ:14,TS-IS-JL-FD:15,MS-NM:14}  the coherency metric \eqref{eq: energy} has recently also been also used in power system analysis and control \cite{ES-BB-DFG:15,FD-MJ-MC-FB:13a,XW-FD-MJ:15a}.

The above metric \eqref{eq: energy} represents a generalized energy in synchronous machines. Indeed, for {$a_{ij} = b_{ij}$} (where $b_{ij}$ are the power line susceptances) and ${s_{i}} = m_{i}$, the metric \eqref{eq: energy} accounts for the potential and kinetic energy in swing mode oscillations. Following the interpretation proposed in \cite{ES-BB-DFG:15}, for {$a_{ij} = g_{ij}$} (where $g_{ij}$ are the power line conductances), the metric \eqref{eq: energy} accounts for the transient resistive losses in the grid lines when linearized around the no-load profile.

Adopting the state representation introduced in \eqref{eq:dynamics}, the performance metric \eqref{eq: energy} can be rewritten as the time-integral $\int_0^\infty y(t)^{\mathsf{T}} y(t)\, \text{d}t$ of the performance output
\begin{equation}
\label{eq:OutputMat}
 y = 
 \underbrace{
 \left[
\setlength\arraycolsep{1.5pt}\begin{array}{cc}
N^{\frac{1}{2}} & 0 \\
0 & S^{\frac{1}{2}}\\
\end{array}\right]
}_{=C} \left[ \setlength\arraycolsep{1.5pt}\begin{array}{cc} 
\theta \\
\omega 
\end{array}\right] \, .
\end{equation}%

In order to model the localization of the disturbances in the grid,
we parametrize the input ${p_\text{in}}$ as 
\begin{equation*}
{p_\text{in}}={V}^{\frac{1}{2}} \eta, \quad {V}=\diag{\{v_i}\}, \quad {v_i} \ge 0,
\end{equation*}
where $V$ is assumed to be known from historical data among other sources.
We therefore obtain the state space model
\begin{equation}
\begin{bmatrix}
\dot{\theta}\\
\dot{\omega}\\
\end{bmatrix}
=
\underbrace{
\begin{bmatrix}
{0} & I \\
-{M}^{-1}L & -{M}^{-1}{D}\\
\end{bmatrix}
}_{=A}
\begin{bmatrix}
\theta\\
\omega\\
\end{bmatrix}
 +
\underbrace {
\begin{bmatrix}
0 \\
M^{-1} {V}^{1/2}
\end{bmatrix}
}_{=B}
\eta.
\label{eq: input}
\end{equation}
In the following, we refer to the input/output map \eqref{eq:OutputMat}, \eqref{eq: input} as $\mathscr{G} = (A,B,C)$.
If the inputs $\eta_i$ are Dirac impulses, then \eqref{eq: energy} measures the squared 
$\mathscr{H}_2$ norm $\|\mathscr{G}\|_{2}^2$ of the system \cite{KZ-JCD-KG:96}.

There are a number of interpretations of the $\mathscr{H}_2$ norm $\|\mathscr{G}\|_{2}^2$ of a power system \cite{ES-BB-DFG:15}. The relevant ones in our context are:
\begin{enumerate}[label=(\alph*), leftmargin=0.05cm, itemindent=0.5cm]
\setlength{\itemsep}{2pt}
\item The squared $\mathscr{H}_2$ norm of $\mathscr{G}$ measures the energy amplification, i.e., the sum of $\mathscr L_2$ norms of the outputs $y_{i}(t)$, for unit impulses at all inputs $\eta_{i}(t) \!=\! \delta(t)$:
\begin{equation*}
{\|\mathscr{G}\|}_{2}^2={\mathlarger{\sum_{i=1}^{n} \int_0^\infty}} y_{i}(t)^{\mathsf{T}} \, y_{i}(t)\, {\text{d}t}.
\end{equation*}
These impulses are of strength ${v^{1/2}_i>0}$ for each node $i \in \until n$ and can model faults or initial conditions.
\item The squared $\mathscr{H}_2$ norm of $\mathscr{G}$ quantifies the steady-state total variance of the output for a system subjected to {unit variance stochastic white noise inputs $\eta_{i}(t)$}:
\begin{equation*}
{\|\mathscr{G}\|}_{2}^2={{\lim_{t \to \infty}}} {\mathbb{E}} \left\{y(t)^{\mathsf{T}} \, y(t)\right\},
\end{equation*}
where $\mathbb{E}$ denotes the expectation operator.
The white noise inputs can model stochastic fluctuations of renewable generation or loads. The matrix ${V^{1/2}}=\textup{diag}\{{ v^{1/2}_i}\}$ quantifies the probability of occurrence of such fluctuations at each node $i$.

\end{enumerate}

In general, the $\mathscr H_{2}$ norm of a linear system can be calculated efficiently by solving a linear Lyapunov equation. In our case an additional linear constraint is needed to account for the marginally stable and undetectable mode $z_0 = [\vectorones[n]^\mathsf{T} \; \vectorzeros[n]^\mathsf{T}]^\mathsf{T}$ corresponding to an absolute angle reference for the grid.
 
\begin{lemma}{{\bf ($\mathscr{H}_2$ norm via observability Gramian)}}
\label{lemma: Observability Gramian}
For the state-space system $(A,B,C)$ defined above, we have that
\begin{equation}
\label{eq:ObservGram}
{\|\mathscr{G}\|}_{2}^2=\textup{Trace}(B^{\mathsf{T}}P B)\, ,
\end{equation}
where the observability Gramian $P \in \real^{2n \times 2n}$ is uniquely defined by the following Lyapunov equation and 
an additional constraint {via} $z_0 = [\vectorones[n]^\mathsf{T} \; \vectorzeros[n]^\mathsf{T}]^\mathsf{T}$:
\begin{align}
\label{eq: Lyap}
& {PA}+{A}^{\mathsf{T}}{P}+{C}^{\mathsf{T}}{C}=0 \,,\\
\label{eq: Lyap constraint}
& P z_0 = \vectorzeros[2n] \,.
\end{align}
\end{lemma}

\begin{proof}
Following the derivation of the $\mathscr H_2$ norm for state-space systems \cite{KZ-JCD-KG:96},
we have ${\|\mathscr{G}\|}_{2}^2=\textup{Trace}(B^{\mathsf{T}}\hat P B)$, where
$\hat P$ is the observability Gramian 
$\hat P={{\int_0^\infty}}e^{{A^{\mathsf{T}}}t}C^{\mathsf{T}} C e^{{A}t}\, \text{d}t$. Note from \eqref{eq:OutputMat} that the 
mode $z_0 = [\vectorones[n]^\mathsf{T} \; \vectorzeros[n]^\mathsf{T}]^\mathsf{T}$ associated with the marginally stable 
eigenvalue of $A$ is not detectable, i.e., it holds that 
$C e^{{A}t} z_0 = C z_0 = \vectorzeros[2n]$ for all $t\geq0$.
Because the remaining eigenvalues of $A$ are stable, the integral is finite.

Next, we show that $\hat P$ is a solution for both \eqref{eq: Lyap} and \eqref{eq: Lyap constraint}. By taking 
the derivative of $e^{{A^{\mathsf{T}}}t}C^{\mathsf{T}} C e^{{A}t}$ with respect to $t$, and then integrating 
from $t=0$ to $t=+\infty$, we obtain
$$
A^\mathsf{T} \hat P + \hat P A = \left[
e^{A^\mathsf{T}t} C^\mathsf{T}  C e^{At}
\right]_0^{\infty}.
$$
Using the fact that $C z_0 = A z_0 = \vectorzeros[2n]$,
we conclude that $\left[ e^{A^\mathsf{T}t} C^\mathsf{T}  C e^{At} \right]_0^{\infty} = -C^\mathsf{T} C$ and 
therefore \eqref{eq: Lyap} holds for $\hat P$. The fact that $\hat P$ satisfies \eqref{eq: Lyap constraint} can be verified by inspection, as
\begin{equation*}
\hat P z_0
 ={{\int_0^\infty}}\!\!\!e^{{A^{\mathsf{T}}}t}C^{\mathsf{T}} C e^{{A}t} z_0\, {\text{d}t}
= {{\int_0^\infty}}\!\!\!e^{{A^{\mathsf{T}}}t}C^{\mathsf{T}} C z_0\, {\text{d}t} = \vectorzeros[2 n].
\end{equation*}

It remains to be shown that $\hat P$ is the unique solution of \eqref{eq: Lyap} and \eqref{eq: Lyap constraint}. As $\Rank{\left(A^{\mathsf{T}}\right)} = 2n-1$, the rank–nullity theorem implies that the kernel of 
$A^{\mathsf{T}}$ is given by a vector $\zeta \in \real^{2n}$.
It can be verified that $A^{\mathsf{T}} \zeta = \vectorzeros[2 n]$ holds for
$\zeta = [(D \vectorones[n])^{\mathsf{T}} \; (M \vectorones[n])^{\mathsf{T}}]^{\mathsf{T}}$. The solutions of  \eqref{eq: Lyap} can hence be parametrized by 
$$
P(\tau) = \hat P + \tau \zeta \zeta^{\mathsf{T}},
$$
for $\tau \in \real$. Finally, \eqref{eq: Lyap constraint} holds if
$ (\hat P + \tau \zeta \zeta^{\mathsf{T}}) z_0 =  \vectorzeros[2 n]$. In combination with 
$\hat P z_0 = \vectorzeros[2 n] $ this implies $\tau = 0$.
With this choice of $\tau$, $P$ equals the positive semidefinite matrix $\hat P$.
\end{proof}

\section{Optimal inertia allocation}
\label{Section: Optimal inertia allocation}
We assume that each node $i \in \{1, \dots, n\}$ has a nonzero\footnotemark\ 
inertia coefficient $\underline{m_i} > 0$
and we are interested in optimally allocating additional virtual inertia in order to minimize the $\mathscr{H}_2$ norm \eqref{eq: energy}, subject to upper bounds $\overline{m_i}$ at each bus (accounting for the available capacity or installation space) and a total budget constraint $m_\text{bdg}$ (accounting for the total cost of the  inertia-emulating devices).
\footnotetext{Observe that the case $m_i=0$ leads to an ill-posed model \eqref{eq:basic} whose number of algebraic and dynamic states depend on the system parameters.}
\noindent This problem statement is summarized as
\begin{subequations}
\label{eq:MinOp1}
\begin{align}
& \underset{P\,,\, m_{i}}{\text{\it minimize}}
& & {\|\mathscr{G}\|}_{2}^2= \trace \left({B}^{\mathsf{T}}{P B}\right) \label{eq:MinOp1a} \\
 & {\text{\it subject to}}
& &  \vectorones[n]^{\sf T}\, {m}\leq{{m}_{\textup{bdg}}} \label{eq:MinOp1b} \\
&& &   {m}_{i} \in [\underline{m_i},\,{\overline{m_i}}]  \label{eq:MinOp1c} \,,\quad i \in \until n \\
& & & {PA} +{A}^{\mathsf{T}} {P} + {C}^{\mathsf{T}}{C}=0, \extraconstraint \label{eq:MinOp1d} 
\,,
\end{align}
\end{subequations}
 where $(A,B,C)$ are the matrices of the input-output system  \eqref{eq:OutputMat}-\eqref{eq: input}. Observe the {bilinear nature of the Lyapunov constraint \eqref{eq:MinOp1d} featuring products of $A$ and $P$}, and recall from \eqref{eq: input} that the decision variables $m_i$ also appear as $m_i^{-1}$ in $A$. Hence, the problem \eqref{eq:MinOp1} is non-convex and typically also large-scale.

In the following, we will provide general lower and upper bounds, a simplified formulation under certain parametric assumptions, a detailed analysis of a two-area power system, and a numerical method to determine locally optimal solutions in the fully general case.

\subsection{ Performance bounds}
\label{subsec: uniform performance bounds}

\begin{theorem}{{\bf (Performance bounds)}}
\label{theorem:  Bounds on performance index}
Consider the power system model \eqref{eq:OutputMat}-\eqref{eq: input}, the squared $\mathscr H_{2}$ norm \eqref{eq:ObservGram}, and the optimal inertia allocation problem \eqref{eq:MinOp1}. 
Then the objective \eqref{eq:MinOp1a} satisfies
\begin{multline}
\label{eq:Bounds}
\frac{\underline  {v}}{2\overline{d}}
\left( \trace(N{L}^\dagger) + 
\sum_{i=1}^n \frac{{s_i}}{m_i} \right) \\
 \leq   {\|\mathscr{G}\|_{2}^{2}} \leq 
\frac{\overline{v}}{2\underline{d}} 
 \left(  \trace (N{L}^\dagger) +
\sum_{i=1}^n \frac{{s_i}}{m_i}\right),
\end{multline}
 where 
 $\underline{v} = \text{\it min}_i{\{v_i}\}$, $\overline{v} = \text{\it max}_i{\{v_i}\}$,
  $\underline{d} = \text{\it min}_i\{d_i\}$, and $\overline{d} = \text{\it max}_i\{d_i\}$.
  \end{theorem}
 
 \begin{proof}
Let us express the observability Gramian $P$ as the block matrix 
\begin{equation*}
{P}=
\begin{bmatrix}
X_1 & X_0 \\
X^{\mathsf{T}}_0 & X_2\\
\end{bmatrix}.
\end{equation*}
With this notation, the squared $\mathscr{H}_2$ norm \eqref{eq:ObservGram}, ${\|\mathscr{G}\|}_{2}^2$ reads as
\begin{equation}
\trace(B^{\mathsf{T}} P B) 
= \trace({V} M^{-2} X_2)
= \sum_{i=1}^{n} \frac{{v_i} X_{2,ii}}{m_i^2},
\label{eq:TraceNew}
\end{equation}
where we use the ring commutativity of the trace and the fact that ${V}^{1/2}$ and $M^{-1}$ are diagonal and therefore commute.
The constraint \eqref{eq:MinOp1d} can be expanded as
\begin{equation}
\label{eq:LyapunovEqn}
\left[\begin{matrix}
X_1 & X_0 \\
{X_0}^{\mathsf{T}} & X_2\\
\end{matrix}\right] A
+
A^{\mathsf{T}} \left[\begin{matrix}
X_1 & X_0 \\
{X_0}^{\mathsf{T}} & X_2 \\
\end{matrix}\right]
+ 
\left[\begin{matrix}
 N& 0 \\
0 &  S\\
\end{matrix}\right]
=0.
\end{equation}
By right-multiplying the equation (1,1) of \eqref{eq:LyapunovEqn} by the Moore-Penrose pseudo-inverse $L^{\dagger}$ of the Laplacian $L$, we obtain 
$$
-X_0 M^{-1} L L^\dagger - L M^{-1} X_0^{\mathsf{T}} L^\dagger = {-}N L^\dagger.
$$
By the constraint \eqref{eq: Lyap constraint} we have that $\left[\vectorones[n]^{\mathsf{T}} ~ \vectorzeros[n]^{\mathsf{T}}\right]P = \left[\vectorzeros[n]^{\mathsf{T}} ~ \vectorzeros[n]^{\mathsf{T}}\right]$ which implies $\vectorones[n]^{\mathsf{T}} X_0 = \vectorzeros[n]^{\mathsf{T}}$.
This fact together with the identity
${L}{L}^\dagger = (I_{n} - \vectorones[n]\vectorones[n]^{\mathsf{T}}/n)$,
implies that $L L^\dagger X_0 = X_0$.
Then, by using the ring commutativity of the trace, and its invariance with respect to transposition of the argument, we obtain
\begin{equation}
\label{eq:TraceSolPa}
2 \trace(M^{-1} X_0) = \trace({N}{L}^\dagger).
\end{equation}
On the other hand,  equation (2,2) of \eqref{eq:LyapunovEqn} implies that
$$
X_0^{\mathsf{T}} + X_0 = X_2 M^{-1}D + D M^{-1}X_2 - S.
$$
Similarly as before we left-multiply by $M^{-1}$, use trace properties and the commutativity of matrices $M^{-1}$, $D$ to obtain
\begin{equation}
\label{eq:TraceSolPb}
2\trace(M^{-1} X_0 - DM^{-2}X_2)= - \trace(M^{-1}S).
\end{equation}
Thus, \eqref{eq:TraceSolPa} and \eqref{eq:TraceSolPb} together deliver
\begin{equation}
\label{eq:TraceSim}
\trace(DM^{-2}X_2) = \frac12 {\trace}(M^{-1}S + N{{L}^\dagger}).
\end{equation}
From (\ref{eq:TraceNew}) we obtain the relations
\begin{equation*}
\underline{v}\sum_{i=1}^{n} \frac{X_{2,ii}}{m_i^2}\leq {\|\mathscr{G}\|}_{2}^2 \leq {\overline{v}} \sum_{i=1}^{n} \frac{X_{2,ii}}{m_i^2}
\,,
\end{equation*}
which can be further bounded as
\begin{equation}
\label{eq:TraceNew2}
\dfrac{\underline{v}}{{\overline{d}}} \sum_{i=1}^{n} \frac{d_i X_{2, ii}}{m_i^2}\leq {\|\mathscr{G}\|}_{2}^2 \leq \dfrac{\overline{v}}{\underline{d}} \sum_{i=1}^{n} \frac{d_i X_{2, ii}}{m_i^2}
\,.
\end{equation}
The structural similarity of \eqref{eq:TraceSim} and \eqref{eq:TraceNew2} allows us to state upper and lower bounds by  rewriting \eqref{eq:TraceNew2} as in \eqref{eq:Bounds}.
\end{proof}

Notice that in the bounds proposed in Theorem~\ref{theorem:  Bounds on performance index}, the network topology described by the Laplacian $L$ enters only as a constant factor, and is decoupled from the decision variables $m_i$. Moreover, in the case $N= L$ (short-range error penalty on angles differences), this offset term becomes just a function of the grid size: $\trace (N L^\dagger) = \trace (L L^\dagger) = n-1$.

Theorem~\ref{theorem:  Bounds on performance index} (and its proof) sheds some light on the nature of the optimization problem that we are considering, and in particular on the role played by the mutual relation between disturbance strengths ${v_i}$, damping coefficients $d_i$, their ratios ${v_i}/d_{i}$, frequency penalty weights $s_i$, and the decision variables $m_i$. These insights are further developed hereafter.

\subsection{Noteworthy cases}
\label{subsec: analytic closed-form results}

In this section, we consider some special choices of the performance metric and some assumptions on the system parameters, which are practically relevant and yield simplified versions of the general optimization problem \eqref{eq:MinOp1}, enabling in most cases the derivation of closed-form solutions.

We first consider the performance index \eqref{eq: energy} corresponding to the effort of primary control. As a remedy to mitigate low-inertia frequency stability issues, additional fast-ramping primary control is often put forward \cite{AU-TB-GA:14}.
The primary control effort can be accounted for by the integral quadratic cost
 \begin{equation}
 	\int_{0}^{\infty} \dot\theta(t)^{\mathsf{T}} D \dot\theta(t) \, {\text{d}t} \,.
	\label{eq: control cost}
 \end{equation}
 Hence, the effort of primary control \eqref{eq: control cost} mimics 
 the $\mathscr H_{2}$ performance where the performance matrices in \eqref{eq:OutputMat} are chosen as $N=0$ and $S=D$. This intuitive cost functions allows an insightful simplification of the optimization problem \eqref{eq:MinOp1}. 
\begin{theorem}{\bf(Primary control effort minimization)} 
\label{theorem: Performance index for primary control cost}
Consider the power system model \eqref{eq:OutputMat}-\eqref{eq: input}, the squared $\mathscr H_{2}$ norm \eqref{eq:ObservGram}, and the optimal inertia allocation problem \eqref{eq:MinOp1}. For a performance output characterizing the effort of primary control \eqref{eq: control cost}:  $S= D$ and $N=0$, the optimization problem \eqref{eq:MinOp1} can be equivalently restated as the convex problem%
\begin{subequations}%
\label{eq:primarycontrol}%
\begin{align}
& \underset{{m}_{i}}{\text{minimize}} \quad \quad
  \sum\limits_{i=1}^{n}  \frac{v_i}{m_i}\label{eq:primarycontrol_cost} \\
 & \text{subject to} \quad \quad
\eqref{eq:MinOp1b}-\eqref{eq:MinOp1c}\,,
\end{align}%
\end{subequations}%
where, we recall, ${v_i}$ describes the strength of the disturbance at node $i$.
\end{theorem}

\begin{proof} 
With $N=0$ and $S = D$, the Lyapunov equation \eqref{eq:LyapunovEqn} together with the constraint \eqref{eq: Lyap constraint} is solved explicitly by
\begin{equation*}
\label{eq:GeneralP}
P=
\begin{bmatrix}
X_1 & X_0 \\
X^{\mathsf{T}}_0 & X_2\\
\end{bmatrix}
=
\frac{1}{2}
\begin{bmatrix}
L & 0 \\
0 & M\\
\end{bmatrix}.
\end{equation*} 
The performance metric as derived in \eqref{eq:TraceNew} therefore becomes
\begin{equation*}
	{\|\mathscr G\|}_{2}^2
	= \sum_{i=1}^{n} \frac{{v_i} X_{2,ii}}{m_i^2}
	= \frac{1}{2} \sum_{i=1}^{n} \frac{v_i}{m_i}.
\end{equation*}
This concludes the proof.
\end{proof}

The equivalent convex formulation \eqref{eq:primarycontrol} yields the following important insights.
First and foremost, the optimal solution to \eqref{eq:primarycontrol} is unique (as long as at least one ${v_i}$ is greater than zero) and also independent of the network topology and the line susceptances. 
It depends solely on the location and strength of the disturbance as encoded in the coefficients ${v_i}$. 
For example, if the disturbance is concentrated at a particular node $i$, that is, ${v_i}\neq 0$ and ${v_j}=0 \, \forall \, j\neq i$, then the optimal solution is to allocate the maximal inertia at node $i$: $m_i=\text{\it min}\{m_\text{bdg}, \overline{m_i}\}$. {If the capacity constraint \eqref{eq:MinOp1c} is relaxed}, the optimal inertia allocation is proportional to the disturbance~${{v_i}^{1/2}}$.

We now consider a different assumption that also allows to derive a similar simplified analysis in other notable cases.

\begin{assumption}{{\bf (Uniform disturbance-damping ratio)}} 
\label{ass: Uniform disturbance-damping ratio}
The ratio $\lambda =\, {v}_{i}/{d}_{i}$ is constant for all $i \in \until n$.
\oprocend
\end{assumption}

Notice that the droop coefficients $d_{i}$ are often scheduled proportionally to the rating of a power source to guarantee fair power sharing \cite{FD-JWSP-FB:14a}.
Meanwhile, it is reasonable to expect that the disturbances due to variable renewable  fluctuations scale proportionally to the size of the renewable power source.
Hence, Assumption \ref{ass: Uniform disturbance-damping ratio} can be justified in many practical cases, including of course the case where both damping coefficients and disturbances are uniform across the grid. Aside from that, Assumption \ref{ass: Uniform disturbance-damping ratio} may be of general interest since it is common in many studies with a {\em spatially invariant} setting \cite{BB-MRJ-PM-SP:12,MS-NM:14,ES-BB-DFG:15}. Under this assumption, we have the following result.
 
\begin{theorem}{\bf (Optimal allocation with uniform disturbance-damping ratio)}
\label{theorem:  Properties of general performance index}
Consider the power system model \eqref{eq:OutputMat}-\eqref{eq: input}, the squared $\mathscr H_{2}$ norm \eqref{eq:ObservGram}, and the inertia allocation problem \eqref{eq:MinOp1}. 
Let Assumption~\ref{ass: Uniform disturbance-damping ratio} hold.
Then the optimization problem \eqref{eq:MinOp1} can be equivalently restated as the convex problem%
\begin{subequations}%
\label{eq:uniformratio}%
\begin{align}
& \underset{{m}_{i}}{\text{minimize}} \quad \quad
  \sum\limits_{i=1}^{n}  \frac{{s_i}}{m_i}\label{eq:uniformratio_cost} \\
 & \text{subject to} \quad \quad
 \eqref{eq:MinOp1b}-\eqref{eq:MinOp1c},
\end{align}%
\end{subequations}%
where we recall that ${s_i}$ is the penalty coefficient for the frequency deviation at node $i$.
\end{theorem}

\begin{proof}
From Assumption \ref{ass: Uniform disturbance-damping ratio}, let $\lambda = {v_i}/d_{i} > 0$ be constant for all $i \in \until n$. Then we can rewrite \eqref{eq:TraceNew} as 
\begin{equation*}
	{\|\mathscr{G}\|}_{2}^2
	= \sum_{i=1}^{n} \frac{v_i X_{2,ii}}{m_i^2}
	= \lambda \sum_{i=1}^{n} \frac{d_i X_{2,ii}}{m_i^2}.
\end{equation*}
This is equal, up to the scaling factor $\lambda$, to the left hand side of \eqref{eq:TraceSim}. We therefore have
\begin{equation}
{\|\mathscr{G}\|}_{2}^2 = \frac{\lambda}{2} \trace(M^{-1}S + N L^\dagger),
\label{eq:norm with assumption 1}
\end{equation}
which is equivalent, up to multiplicative factors and constant offsets, to the cost of the optimization problem \eqref{eq:uniformratio_cost}.
\end{proof}

Again, as in Theorem \ref{theorem: Performance index for primary control cost}, Theorem \ref{theorem:  Properties of general performance index} reduces the original optimization problem to a simple convex problem for which the optimal inertia allocation is {\em independent} of the network topology.
 Indeed, the physical intuition of Assumption~\ref{ass: Uniform disturbance-damping ratio} is that the disturbance is dissipated at every node in the same proportion, and thus network effects are negligible. 

This setting also allows us to highlight that the cost function for inertia allocation  needs to be chosen insightfully. For example, consider a frequency penalty $S$  proportional to the inertia coefficients, $S = cM$ for some $c \ge 0$ (including $c=0$):
 \begin{equation*}
	\mathlarger{\int_0^\infty} \bigg\{\mathlarger{ \sum_{i,j =1}^{n}} {a_{ij}} (\theta_{i}(t)-\theta_{j}(t))^2 + c \,{\mathlarger \sum_{i=1}^{n}} m_{i}\, \omega_{i}^2(t) \, \bigg\}\, {\text{d}t}\,.
\end{equation*}
This choice penalizes the variation in kinetic energy as it decays to zero -- a standard penalty in power systems.
The subsequent corollary shows that this cost function is independent of (and thus not meaningful for) the inertia allocation. The physical rationale is that kinetic energy is dissipated anyways.

\begin{corollary}{\bf (Kinetic energy penalization with uniform disturbance-damping ratio)}
Let Assumption~\ref{ass: Uniform disturbance-damping ratio} hold, and let the penalty on the frequency deviations be proportional to the allocated inertia, that is, $S = cM$.
Then the performance metric $\|\mathscr{G}\|_{2}^{2}$ is independent of the inertia allocation, and assumes the form
$$
\|\mathscr{G}\|_{2}^{2}
 = \frac{\lambda}{2} \left( c \,n + \trace(NL^\dagger)\right),
$$
where
$\lambda = {v_i}/d_{i} > 0$, for all $i \in \until n$, is the uniform disturbance-damping ratio.
\end{corollary}

\subsection{Explicit results for a two-area network}
\label{subsec: two-area}

In this subsection, we focus on a two-area power grid as in \cite{AU-TB-GA:14} to obtain some insight on the nature of this optimization problem. We also highlight the prominent role of the ratios ${v_i}/d_{i}$ as in Assumption~\ref{ass: Uniform disturbance-damping ratio} and the bounds \eqref{eq:Bounds}.

 In the case of a two-area system, it is possible to derive an analytical solution $P(m)$ of the Lyapunov equation \eqref{eq:MinOp1d}, as a closed-form function of the  vector of inertia allocations $m_i$. We thus obtain an explicit expression for the cost \eqref{eq:MinOp1a} as
  \begin{equation}
\|\mathscr{G}\|_2^2 = f(m) := \trace\left(B({{m}})^{\mathsf{T}} {P(m)} B({{m}}) \right),
\label{eq:f}
\end{equation}
 where {in the two-area case} $f(m)$ {reduces to a rational function of polynomials of orders 4 in the numerator and the denominator, in terms of inertia coefficients $m_{i}$.}
As the explicit expression is more convoluted than insightful, we will not show it here but only report the following observations:
\begin{enumerate}[label=(\alph*), leftmargin=0.05cm, itemindent=0.5cm]
\setlength{\itemsep}{2pt}
\item The problem \eqref{eq:MinOp1} admits a unique minimizer.
	
\item For sufficiently large bounds $\overline{m_i}$, the budget constraint \eqref{eq:MinOp1b} becomes active, that is, the optimizers satisfy $m_{1}^{*}+ m_{2}^{*}={m}_{\text{bdg}}$.
	In this case,  $m_{2} = m_{\textup{bdg}} - m_{1}$ can be eliminated, and  \eqref{eq:MinOp1} can be reduced to a scalar problem.

\item Identical $v_{i}/d_{i}$ ratios and frequency penalties $s_i$ result in identical optimal allocations $m_{1}^{*} = m_{2}^{*}$ (as predicted by Theorem \ref{theorem:  Properties of general performance index}), if capacity constraints are absent. If ${v_i}/d_i > {v_j}/d_j$, then $m_i^* > m_j^*$ (see the example in Figure~\ref{fig:2nodetrace}, where we eliminated $m_{2}^{*}={m}_{\text{bdg}}-m_{1}^{*}$).

\item For sufficiently uniform $v_i/d_{i}$ ratios, the problem \eqref{eq:MinOp1} is strongly convex. We observe that the cost function $f(m)$ is fairly flat over the feasible set (see Figure~\ref{fig:2nodetrace}).
		
\item For strongly dissimilar  $v_i/d_{i}$ ratios, we observe a less flat cost function. If disturbance affects only one node, for example, $v_1=1$ and $v_2=0$, then strong convexity is lost. 
\end{enumerate}

\begin{figure}[tb]
\centering
%
%
\definecolor{mycolor2}{HTML}{253F5B}%
\definecolor{mycolor1}{HTML}{D59B2D}%
\begin{tikzpicture}

\begin{axis}[%
width=2.9in,
height=1.6in,
at={(2.17in,1.077in)},
scale only axis,
xmin=0,
xmax=10,
xlabel style={font={\small\color{black}}},
xlabel={\small$m_1$},
ymin=0,
ymax=7,
ylabel style={font={\small\color{black}}},
ylabel={\small${f(m_1)}$},
axis background/.style={fill=white},
xmajorgrids,
ymajorgrids,
xtick={0, 2,4,6,8,10},
ytick={1,2,3,4,5,6},
yticklabel style = {font=\footnotesize,xshift=0ex},
xticklabel style = {font=\footnotesize,yshift=0ex},
legend entries={${v_1}/{d_1}= {v_2}/{d_2}$, ${v_1}/{d_1}\neq{v_2}/{d_2}$},
legend style={legend cell align=left, align=left, draw=none, font=\small},
legend style={at={(0.5,0.9)},anchor=north}
]
\addlegendimage{no markers, mycolor1, very thick}
\addlegendimage{no markers,mycolor2, very thick}

\addplot [color=mycolor1, line width=3pt]
  table[row sep=crcr]{%
0.380000000000001	7.00411697122224\\
0.4	6.75781249999999\\
0.42	6.53500596480764\\
0.44	6.3324933434766\\
0.460000000000001	6.14762783702488\\
0.48	5.97820378151261\\
0.5	5.82236842105263\\
0.52	5.67855404089582\\
0.540000000000001	5.54542518205308\\
0.56	5.42183716707022\\
0.58	5.30680320667692\\
0.609999999999999	5.14845100298539\\
0.640000000000001	5.00500801282051\\
0.67	4.87447209291165\\
0.699999999999999	4.75518433179724\\
0.73	4.64575852285322\\
0.76	4.54502734107997\\
0.790000000000001	4.45200078340824\\
0.82	4.36583373186673\\
0.85	4.28580038572806\\
0.880000000000001	4.21127392344497\\
0.91	4.14171018750227\\
0.94	4.07663449344793\\
0.970000000000001	4.01563088673493\\
1	3.95833333333333\\
1.03	3.904418449849\\
1.07	3.83730285397327\\
1.11	3.77510032529718\\
1.15	3.71729918938836\\
1.19	3.66345658581253\\
1.23	3.61318727925022\\
1.27	3.56615458505831\\
1.31	3.52206295733448\\
1.35	3.48065189466924\\
1.39	3.44169089815256\\
1.44	3.3961253894081\\
1.49	3.35371828642181\\
1.54	3.31416367320623\\
1.59	3.27719415341126\\
1.64	3.24257497957755\\
1.69	3.21009922457437\\
1.75	3.1737012987013\\
1.81	3.13984933114767\\
1.87	3.10830110306451\\
1.93	3.07884459810852\\
2	3.046875\\
2.07	3.01724098543414\\
2.14	2.98971784261967\\
2.22	2.96059716528868\\
2.3	2.93372388481084\\
2.38	2.9088781181764\\
2.47	2.88311329580463\\
2.56	2.85943800403226\\
2.66	2.83533680932577\\
2.76	2.81332572663944\\
2.87	2.79128345656328\\
2.98	2.7712886479665\\
3.1	2.7515778401122\\
3.22	2.73384680921233\\
3.35	2.71665637975536\\
3.49	2.70026771684734\\
3.63	2.68587743425406\\
3.78	2.67247869136937\\
3.94	2.66029426546716\\
4.1	2.65011368334022\\
4.27	2.64133520114767\\
4.44	2.63452751312464\\
4.62	2.62935716699657\\
4.8	2.62620192307692\\
4.98	2.625012000192\\
5.16	2.62576878723814\\
5.34	2.62848411052708\\
5.52	2.63320069875776\\
5.7	2.63999388004896\\
5.87	2.64841594103065\\
6.03	2.65823746506761\\
6.19	2.67003390024551\\
6.34	2.68303640688835\\
6.49	2.69809404343303\\
6.63	2.71418525182271\\
6.76	2.73107056760903\\
6.89	2.75002741752575\\
7.01	2.76956533666668\\
7.12	2.78938436329588\\
7.23	2.81123140644427\\
7.33	2.83304528105216\\
7.43	2.85692730072113\\
7.52	2.88038349347975\\
7.61	2.90590516222324\\
7.69	2.9305114586324\\
7.77	2.95711991620063\\
7.85	2.98594652644053\\
7.92	3.01318473193473\\
7.99	3.0425041563148\\
8.06	3.07412511831368\\
8.12	3.10325175558118\\
8.18	3.13443738413176\\
8.24	3.16788945278023\\
8.3	3.20384479092842\\
8.35	3.23591453456723\\
8.4	3.27008928571428\\
8.45	3.3065709104791\\
8.5	3.34558823529412\\
8.55	3.3874016938899\\
8.6	3.43230897009967\\
8.64	3.47069035947712\\
8.68	3.51146837033934\\
8.72	3.55486668577982\\
8.76	3.60113787008396\\
8.8	3.65056818181819\\
8.84	3.70348338274302\\
8.88	3.7602557915058\\
8.92	3.82131290483309\\
8.96	3.88714800824176\\
8.99	3.94000071586692\\
9.02	3.99613670301823\\
9.05	4.05586653096831\\
9.08	4.11954127561769\\
9.11	4.18755935569012\\
9.14	4.26037479008702\\
9.17	4.33850724599598\\
9.2	4.42255434782609\\
9.23	4.51320686355897\\
9.26	4.61126758507967\\
9.29	4.71767499507269\\
9.32	4.83353319868719\\
9.35	4.96015014397367\\
9.38	5.099086938579\\
9.4	5.19946808510639\\
9.42	5.30680320667692\\
9.44	5.42183716707022\\
9.46	5.54542518205309\\
9.48	5.67855404089582\\
9.5	5.82236842105263\\
9.52	5.9782037815126\\
9.54	6.14762783702488\\
9.56	6.33249334347662\\
9.58	6.53500596480764\\
9.6	6.75781249999999\\
9.62	7.00411697122222\\
};

\addplot [color=mycolor2, line width=3pt]
  table[row sep=crcr]{%
0.0500000000000007	7.18500517259066\\
0.0600000000000005	6.14384720309256\\
0.0700000000000003	5.400309354286\\
0.0800000000000001	4.84278448631288\\
0.0899999999999999	4.40926863392843\\
0.0999999999999996	4.06255941919967\\
0.109999999999999	3.77898260273188\\
0.119999999999999	3.54275535928207\\
0.130000000000001	3.34295111847372\\
0.140000000000001	3.17176518806257\\
0.15	3.02347412789595\\
0.16	2.89378535863744\\
0.17	2.77941632035345\\
0.18	2.67781391157092\\
0.19	2.58696252652736\\
0.210000000000001	2.4313700321601\\
0.23	2.30302403245107\\
0.25	2.19538646834306\\
0.27	2.10385649112644\\
0.290000000000001	2.02510245893178\\
0.31	1.9566525152259\\
0.33	1.89663404785533\\
0.35	1.84360247641814\\
0.369999999999999	1.79642556805625\\
0.390000000000001	1.754203348389\\
0.41	1.71621145314946\\
0.43	1.68186028866688\\
0.449999999999999	1.65066508288753\\
0.470000000000001	1.62222358294381\\
0.49	1.59619921454348\\
0.52	1.5610859671482\\
0.550000000000001	1.52999756231772\\
0.58	1.50231125734521\\
0.609999999999999	1.4775268359727\\
0.640000000000001	1.45523789149825\\
0.67	1.43511082487261\\
0.710000000000001	1.41116667212819\\
0.75	1.39004606109869\\
0.790000000000001	1.37132347784967\\
0.83	1.35465548508089\\
0.880000000000001	1.33628770797218\\
0.93	1.32025409858322\\
0.98	1.30620307732259\\
1.04	1.29156395629873\\
1.1	1.27898294100601\\
1.17	1.26649561646476\\
1.24	1.25599970147392\\
1.32	1.24605315835163\\
1.41	1.23705334281046\\
1.5	1.22999549470138\\
1.6	1.22406319693966\\
1.71	1.21949111891104\\
1.83	1.2164663778104\\
1.96	1.2151391571349\\
2.1	1.21563377251545\\
2.26	1.21828399145554\\
2.43	1.22320016325769\\
2.61	1.23046695263057\\
2.8	1.24018809715925\\
3	1.2524931438544\\
3.2	1.2667795330139\\
3.41	1.28381337492631\\
3.62	1.30287601510415\\
3.83	1.32395266834954\\
4.04	1.34706175216776\\
4.25	1.37225016634138\\
4.45	1.39823842049621\\
4.65	1.42626150398965\\
4.84	1.45486176223234\\
5.03	1.48549664165738\\
5.21	1.51650742319888\\
5.39	1.54957578317782\\
5.56	1.58282126063666\\
5.73	1.61815971999861\\
5.89	1.65346826201825\\
6.04	1.68850658067919\\
6.19	1.72555934079297\\
6.33	1.76209810403835\\
6.47	1.8006759452728\\
6.6	1.83847190715499\\
6.73	1.87832897807617\\
6.85	1.91710630753084\\
6.97	1.95796056168432\\
7.08	1.99739928540853\\
7.19	2.03891847390822\\
7.29	2.0786378428597\\
7.39	2.12041975279335\\
7.49	2.16446988738933\\
7.58	2.20625270078717\\
7.67	2.25027286931924\\
7.75	2.29147311255279\\
7.83	2.33482580659438\\
7.91	2.38055795924491\\
7.98	2.42273244024379\\
8.05	2.46713612876183\\
8.12	2.51400629417571\\
8.18	2.55635131017778\\
8.24	2.60090965098474\\
8.3	2.64791010194033\\
8.36	2.69761588924293\\
8.42	2.75033128243354\\
8.47	2.79681309582029\\
8.52	2.84586486930515\\
8.57	2.89775573270917\\
8.62	2.95279453118502\\
8.67	3.01133731901442\\
8.71	3.06096945308426\\
8.75	3.11334874665369\\
8.79	3.1687494618137\\
8.83	3.22748376776819\\
8.87	3.28990845927248\\
8.91	3.35643315389167\\
8.95	3.42753036238495\\
8.99	3.5037479514508\\
9.02	3.56465171102173\\
9.05	3.62910192158224\\
9.08	3.69744827581883\\
9.11	3.77008780545431\\
9.14	3.84747313957795\\
9.17	3.93012255375494\\
9.2	4.0186322799871\\
9.23	4.11369169411322\\
9.26	4.21610219721179\\
9.29	4.32680088359061\\
9.32	4.44689047356611\\
9.35	4.57767753503344\\
9.38	4.7207218013131\\
9.41	4.87790053478833\\
9.43	4.99164492306172\\
9.45	5.1134712880764\\
9.47	5.24429729278614\\
9.49	5.38518461381863\\
9.51	5.5373683317187\\
9.53	5.70229382504298\\
9.55	5.88166350283793\\
9.57	6.07749657869493\\
9.59	6.29220633960414\\
9.61	6.52870118979046\\
9.63	6.79051846547083\\
9.65	7.08200412890097\\
};

\end{axis}
\end{tikzpicture}%
\caption{Cost function profiles for identical and weakly dissimilar $v_i/d_i$ ratios for the two-area case.}
\label{fig:2nodetrace}
\smallskip
{$d_1=1 < d_2=2$, ${m}_{\text{bdg}}=25$, $a_{12}=1$}\\
\smallskip
     \input{budget.tex}
\caption{Optimal inertia allocation for a {two-area} system with non-identical damping coefficients $d_{i}$, and disturbances inputs varying from $(v_1,v_2) = (0,1)$ to $(v_1,v_2) = (1,0)$. We choose $d_1=1<d_2=2$, $m_{\text{bdg}}=25$, and $a_{12}=1$ as the system parameters.}
\label{fig_simb}
\end{figure}
From the above facts, we conclude that the input scaling factors $v_i$ play a fundamental role in the determination of the optimal inertia allocation.
To obtain a more complete picture, we linearly vary the disturbance input matrices from $({v_1,v_2}) = (0,1)$ to $({v_1,v_2}) = (1,0)$, that is, from a disturbance localized at node 2 to a disturbance localized at node 1.
For each value of $({v_1,v_2})$, we compute the optimal inertia allocation for the cost function with $N=L$ and $S=I_2$.
The resulting optimizers are displayed in Figure~\ref{fig_simb} showing that inertia is allocated dominantly at the site of the disturbance, which is in line with previous case studies \cite{TSB-TL-DJH:15,AU-TB-GA:14}. Notice also that depending on the value of the budget ${m}_{\text{bdg}}$, the capacity constraints $\overline{m_i}$, and the ${v_i/d_i}$ ratios, the budget constraint may be active or not. Thus, perhaps surprisingly, sometimes not all inertia resources are allocated. Overall, the two-area case paints a surprisingly complex picture.

\subsection{A computational method for the general case}
\label{subsec: PerturbGrad}

In Subsections {\ref{subsec: analytic closed-form results} and \ref{subsec: two-area}}, we considered a subset of scenarios and cost functions that allowed the derivation of tractable reformulations and solutions of the inertia allocation problem \eqref{eq:MinOp1}. In this section, we consider the optimization problem in its full generality.
As in Section~\ref{subsec: two-area}, we denote by $P(m)$ the solution to the Lyapunov equation~\eqref{eq:MinOp1d}, and express the cost $\|\mathscr{G}\|_2^2$ as a function $f(m)$ of the vector of inertia allocations $m_i$.
In the following, we derive an efficient algorithm for the computation of the explicit gradient $\nabla f(m)$ of $f(m)$ in \eqref{eq:f}.

In general, most computational approaches can be sped up tremendously if an explicit gradient is available. In our case, an additional significant benefit of having a gradient $\nabla f(m)$ of $f(m)$ is that the large-scale set of nonlinear (in the decision variables) Lyapunov equations \eqref{eq:MinOp1d} 
can be eliminated and included into the gradient information. In the following, we provide an algorithm that achieves so, using the routine $\text{\bf Lyap}(A,Q)$, which returns the matrix $P$ that solves $PA+A^{\mathsf{T}}P + Q=0$ together with $Pz_0 = \vectorzeros[2n]$.

\begin{algorithm}[!ht]
\SetKw{KwInput}{Input}
\SetKw{KwOutput}{Output}
\KwInput{\rm current value $m$ of the decision variables\\}
\KwOutput{\rm numerical evaluation $g$ of the gradient $\nabla f(m)$}
\BlankLine
$A^{(0)} \leftarrow
\left[\begin{matrix}
0 & I \\
-M^{-1} L & -M^{-1}D
\end{matrix}\right]
\smallskip
$\;
$B^{(0)} \leftarrow
\left[\begin{matrix}
0 \\
M^{-1}\, {V}^{1/2}
\end{matrix}\right]
\smallskip
$\;
$P^{(0)} \leftarrow
\text{\bf Lyap}\left(A^{(0)}, C^{\mathsf{T}} C\right)$\;
\For{$i = 1, \ldots, n$}{
$\Phi \leftarrow \mathbbold e_{i}\mathbbold e_{i}^{\mathsf{T}}
\smallskip
$\;
$A^{(1)} \leftarrow
\left[\begin{matrix}
0 & 0 \\
\Phi M^{-2} L &  \Phi M^{-2} D
\end{matrix}\right]
\smallskip
$\;
$B^{(1)} \leftarrow
\left[\begin{matrix}
0 \\
-\Phi M^{-2}\, {V}^{1/2}
\end{matrix}\right]
$\;
$P^{(1)} \leftarrow 
\text{\bf Lyap}\left(A^{(0)}, P^{(0)} A^{(1)} + {A^{(1)}}^{\mathsf{T}}P^{(0)}\right)$\;
$g_i
\leftarrow
\text{\bf Trace}\left( 2 {B^{(1)}}^{\mathsf{T}} P^{(0)} B^{(0)}
+  {B^{(0)}}^{\mathsf{T}} P^{(1)} B^{(0)}\right)$\;
}
\caption{Gradient computation}
\label{algorithm}
\end{algorithm}

\begin{theorem}{\bf(Gradient computation)} 
\label{theorem: gradient computation}
Consider the objective function \eqref{eq:f}, where $P(m)$ is a function of $m$ via the Lyapunov equation \eqref{eq:MinOp1d}.
The objective function is differentiable for $m \in \real^{n}_{>0}$, and its gradient at $m$ is given by Algorithm \ref{algorithm}.
\end{theorem}

The proof of Theorem \ref{theorem: gradient computation} is partially inspired by \cite{MF-XZ-FL-MRJ:14} and relies on a perturbation analysis of the Lyapunov equation \eqref{eq:MinOp1d} combined with Taylor and power series expansions.

\begin{proof}
In order to compute the gradient of \eqref{eq:f} at ${{m}}  \in \real^{n}_{>0}$, we make use of the relation
\begin{equation}
\nabla_{{\mu}} f({{m}}) = {\nabla  f({{m}})}^{\mathsf{T}} {\mu}\, ,
\label{eq: gradient 1}
\end{equation}
where $\nabla_{{\mu}} f({{m}})$ is the directional derivative of $f$ in the direction $\mu \in \real^n$, defined as 
\begin{equation}
\nabla_{{\mu}} f(m) = \lim_{\delta \to 0} \frac{f(m + \delta \mu)- f(m)}{\delta},
\label{eq: gradient 2}
\end{equation}
whenever this limit exists. 
From \eqref{eq:f} we have that 
\begin{equation}
f(m + \delta \mu) = 
\trace\left(
	B(m+\delta \mu)^{\mathsf{T}} P B(m+\delta \mu) 
	\right) \, ,
	\label{eq:fdirectional}
\end{equation}
where $P$ is a solution of the Lyapunov equation
\begin{equation}
P A(m+\delta \mu) + {A(m+\delta \mu)}^{\mathsf{T}} {P} + {C}^{\mathsf{T}} {C}=0 \,
\label{eq: Lyap-delta}
\end{equation}
and where by $A(m+\delta \mu)$ we denote the system matrix defined in \eqref{eq: input}, evaluated at $m+\delta \mu$. 
The matrices $A(m+\delta \mu)$ and $B(m+\delta \mu)$ viewed as functions of scalar $\delta$ can thus be expanded in a Taylor series around $\delta=0$ as
\begin{equation}
\label{eq: matrix expansions}
\begin{split}
A(m+\delta \mu) &= A^{(0)}_{(m, \mu)} + A^{(1)}_{(m, \mu)} \delta + {\mc O(\delta^2)} \,, \\
B({{m}} +\delta \mu) &= B^{(0)}_{(m, \mu)}+ B^{(1)}_{(m, \mu)} \delta + {\mc O(\delta^2)}\,, 
\end{split}
\end{equation}
with coefficients $A^{(i)}_{(m, \mu)}$ and $B^{(i)}_{(m, \mu)}$, $i \in \{0,1\}$.
\smallskip
To compute the coefficients of the Taylor expansion in \eqref{eq: matrix expansions}, we recall the scalar series expansion of $1/(m_i+\delta \mu_i)$ around $\delta = 0$:
$$
\frac{1}{(m_i+\delta \mu_i)}
=
\frac{1}{m_i} - \frac{\delta \mu_i}{m_i^2} + \mc O(\delta^2).
$$
Using the shorthand ${{\Phi}}=\diag({\mu_i})$, we therefore have
\begin{align*}
A^{(0)}_{(m, \mu)}
&=
\begin{bmatrix}
0 & I \\ -M^{-1}L & -M^{-1}D
\end{bmatrix},\,
B^{(0)}_{(m, \mu)}=
\begin{bmatrix}
0 \\ M^{-1}{V}^\frac{1}{2}
\end{bmatrix},\\
A^{(1)}_{(m, \mu)}
&=
\begin{bmatrix}
0 & 0 \\ \Phi M^{-2}L & \Phi M^{-2}D
\end{bmatrix},\,
B^{(1)}_{(m, \mu)}=
\begin{bmatrix}
0 \\ -\Phi M^{-2}{V}^\frac{1}{2}
\end{bmatrix}.
\end{align*}
 
Accordingly, the solution to the {Lyapunov equation \eqref{eq: Lyap-delta} can be expanded in a power series as} 
\begin{equation}
\label{eq:5}
P = P({{m}} +\delta \mu)= P^{(0)}_{(m, \mu)}+ P^{(1)}_{(m, \mu)} \delta + \mc O(\delta^2),
\end{equation}
and therefore the Lyapunov equation \eqref{eq: Lyap-delta} becomes 
\begin{multline*}
(P^{(0)}+{\delta} P^{(1)}+{\mc O(\delta^2)}) (A^{(0)}+{\delta} A^{(1)}+{ \mc O(\delta^2)})
+ \\
(A^{(0)}+{\delta} A^{(1)}+{\mc O(\delta^2)})^{\mathsf{T}} (P^{(0)}+ {\delta} P^{(1)}+{\mc O(\delta^2)})+C^{\mathsf{T}}C=0 ,
\end{multline*}
where we dropped the subscript $(m,\mu)$ for readability.
By collecting terms associated with powers of $\delta$, we obtain two Lyapunov equations determining $P^{(0)}$ and $P^{(1)}$:
\begin{subequations}
\label{lyap0}
\begin{equation}
\label{lyap1}
P^{(0)} A^{(0)}+{A^{(0)}}^{\mathsf{T}} P^{(0)}+C^{\mathsf{T}} C=0 \, ,
\end{equation}
\begin{equation}
\label{lyap2}
P^{(1)} A^{(0)}+{A^{(0)}}^{\mathsf{T}} P^{(1)}+(P^{(0)} A^{(1)}+{A^{(1)}}^{\mathsf{T}} P^{(0)})=0 \, .
\end{equation}
\end{subequations}
 By the same reasoning as used for equation \eqref{eq: Lyap}, the first Lyapunov equation \eqref{lyap1} is feasible with a positive semidefinite $P^{(0)}$ satisfying 
$P^{(0)} z_0 = \vectorzeros[2n]$. 
The second Lyapunov equation \eqref{lyap2} is feasible by analogous arguments.
Finally, by using \eqref{eq:fdirectional} together with \eqref{eq: matrix expansions} and \eqref{eq:5}, we obtain 
\begin{equation*}
\label{eq:7}
f(m+\delta\mu)= f^{(0)}_{(m, \mu)} + f^{(1)}_{(m, \mu)} \delta + {\mc O(\delta^{2})} \, ,
\end{equation*}
where $f^{(0)}_{(m, \mu)} = f(m)$ and
\begin{equation}
\begin{split}
f^{(1)}_{(m, \mu)} 
&= 
\trace\Big( 
2 \,{{B^{(1)}_{(m, \mu)}}}^{\mathsf{T}} P^{(0)}_{(m, \mu)} B^{(0)}_{(m, \mu)} \\
& \quad + {B^{(0)}_{(m,\mu)}}^{\mathsf{T}} P^{(1)}_{(m, {\mu})} B^{(0)}_{(m,\mu)}\Big)\, .
\end{split}
\label{eq:9}
\end{equation}
From \eqref{eq: gradient 2}, it follows that 
$\nabla_\mu f(m) = f^{(1)}_{(m, \mu)}$ as
 defined in \eqref{eq:9}, thereby implicitly establishing differentiability of $f({{m}})$. 

This concludes the proof, as the algorithm computes each component of the gradient $\nabla f(m)$ by using the relation \eqref{eq: gradient 1} with the special choice of $\mu = \mathbbold e_{i}$ for $i \in \until n$. 
\end{proof}

\subsection{The planning problem: Economic allocation of resources}
\label{Subsection: l1_norm}

In this subsection, we focus on the planning problem of optimally allocating virtual inertia when economic reasons suggest that only a limited number of virtual inertia devices should be deployed (rather than at every grid bus). 
Since this problem is generally combinatorial, we solve a modified optimal allocation problem, where an additional $\ell_{1}$-regularization penalty is imposed, in order to promote a sparse solution  \cite{EJC-MBW-SPB:08}.

The regularized optimal inertia allocation problem is then \begin{subequations}
\label{eq:SparMin}
\begin{align}
& \underset{P\,,\, m_{i}}{\text{\it minimize}}
& & J_{\gamma}(m, P)=\|\mathscr{G}\|_{2}^{2} +\gamma \|m-\underline{m}\|_{1} \label{eq:SparMina} \\
 & {\text{\it subject to}}
& & \eqref{eq:MinOp1b}-\eqref{eq:MinOp1d}\,,
\end{align}
\end{subequations}
where $\gamma \geq 0$ trades off the sparsity penalty and the original objective function. 

As in \eqref{eq:MinOp1c} the allocations $m_i$ are lower bounded by a positive $\underline{m_i}$, the objective \eqref{eq:SparMina} can be rewritten as:
\begin{equation}
\label{eq:SparMinRed}
{J_{\gamma}(m, P)}=\trace \left({B}^{\mathsf{T}}{P B}\right) + \gamma\, \vectorones[n]^{\sf T}\,(m-\underline{m})\, . 
\end{equation}
Observe that the regularization term in the cost \eqref{eq:SparMinRed} is linear and differentiable. Thus,
 problem \eqref{eq:SparMinRed} fits well into our gradient computation algorithm, and a solution can be determined within the fold of Algorithm \ref{algorithm} by incorporating the penalty term.  Likewise, our analytic results in Section \ref{subsec: analytic closed-form results} can be re-derived for the cost function \eqref{eq:SparMinRed}. We highlight the utility of the performance-sparsity trade-off \eqref{eq:SparMinRed}  in Section \ref{Section: Case study}. 
 
\subsection{The min-max problem: optimal robust allocation}
\label{Subsection: Robustness}

Thus far we have assumed knowledge of the disturbance strengths encoded in the matrix $V$. While empirical disturbance distributions from historical data  are generally available to system operators, the truly problematic and devastating faults in power systems are rare events that are poorly predicted by any (empirical) distribution. Given this inherent uncertainty, it is desirable to obtain an inertia allocation profile which is optimal even in presence of the most detrimental disturbance. This problem belongs to the domain of robust optimization or a zero-sum game between the power system operator and the adversarial disturbance. The robust inertia allocation problem can then be formulated as the min-max optimization problem%
\begin{subequations}%
\label{eq:genminprobfb}%
\begin{align}%
& \underset{{m}_{i}}{\text{\it minimize}} \quad \underset{v_i}{\text{\it maximize}} \quad f(m,v)
\label{eq:genminprob1afb}\\
 & \text{\it subject to} \qquad v \in \mathbb{V}, \\
 & \qquad \qquad \qquad   \eqref{eq:MinOp1b}-\eqref{eq:MinOp1c} \label{eq:genminprob4bfb},
\end{align}
\end{subequations}
\noindent {where $f(m,v)$ = $\trace(B(m,v)^{\sf T}P(m)B(m,v))$} and where $\mathbb{V}$ is convex hull of the set of possible disturbances with a non-empty interior. As a special instance, consider
\begin{equation}
\mathbb{V} = \left\{v \in \real^{n}:\, \vectorones[n]^{\sf T} v \leq{{v}_{\textup{bdg}}},\, 0 \leq {v}_{i} \right\} \,,
\label{eq:genminprob2bfb}
\end{equation}
where we normalized the disturbances by ${v}_{\textup{bdg}}>0$.

Recall from \eqref{eq: input}, \eqref{eq:ObservGram} that the objective $f(m,v)$ is linear in $v$, and we can write it as $f(m,v)=v^{\sf T} g(m)$, where $g_i(m)={d f(m,v)}/{d v_i}={{P(m)_{2,2}}_i}/{m_i^2}$ for $i \in \until n$. Hence, by strong duality, we can rewrite the inner maximization problem
\begin{subequations}
\begin{align}
& \underset{{v}_{i}}{\text{\it maximize}} \quad v^{\sf T}g(m) \label{eq:genminprob3bfb}\\
 & \text{\it subject to} \quad \quad \eqref{eq:genminprob4bfb}, \, \eqref{eq:genminprob2bfb}
\end{align}%
\end{subequations}%
as the equivalent dual minimization problem
\begin{subequations}
\begin{align}
& \underset{\chi, \, \mu_i}{\text{\it minimize}} \quad v_{\textup{bdg}}\, {\chi} \label{eq:dualprob1}\\
 & \text{\it subject to} \quad \quad \chi \geq 0, \quad \mu_{i} \geq 0 \,,\quad \forall i,\label{eq:dualprob2}\\
 & \qquad \qquad \qquad   g_i(m)+\mu_i =\chi \,,\quad \forall i\label{eq:dualprob3}\,,
\end{align}%
\end{subequations}%
where $\chi$ and $\mu_i$ are the dual variables associated with the constraints \eqref{eq:genminprob2bfb}. The min-max problem \eqref{eq:genminprobfb} is then equivalent\,to:%
\begin{subequations}%
\label{eq:genminprobfb-2}%
\begin{align}%
& \underset{{m}_{i},\, {\chi},\,\mu_i}{\text{\it minimize}} \quad v_{\textup{bdg}}\, {\chi}\\
 & \text{\it subject to} \quad \quad \eqref{eq:MinOp1b}-\eqref{eq:MinOp1c}, \\
& \qquad \qquad \qquad   \eqref{eq:dualprob2}-\eqref{eq:dualprob3}.
\end{align}%
\end{subequations}%
The minimization problem \eqref{eq:genminprobfb-2} has a convex objective and constraints, barring \eqref{eq:dualprob3}. However, we already have the gradient of the individual elements, $d g_i(m)/dm$ which can be computed from Algorithm~\ref{algorithm} as $d f(m,v)/dm$ by substituting $v_i=1$ and $v_{j}=0\, \forall\, j \neq i$. The availability of the gradient of this set of non-linear equality constraints considerably speeds up the computation of the minimizer.

By direct inspection or computation (see Section \ref{Section: Case study})
we observe that the robust optimal allocation profile tends to make the cost \eqref{eq:genminprob1afb} indifferent with respect to the location of the disturbance, as is customary for similar classes of min-max (adversarial) optimization problems.

For the special case of {\it primary control effort minimization} as in Theorem \ref{theorem: Performance index for primary control cost}, the min-max problem \eqref{eq:genminprobfb} simplifies to
\begin{subequations}%
\label{eq:genminprobfb-3}%
\begin{align}%
& \underset{{m}_{i},\, {\chi},\,\mu_i}{\text{\it minimize}} \quad v_{\textup{bdg}}\, {\chi}\\
& \text{\it subject to} \quad \quad {\chi} \geq 0, \quad \mu_{i} \geq 0 \,,\quad \forall i,\label{eq:dualprobred2}\\
& \qquad \qquad \qquad \eqref{eq:MinOp1b}-\eqref{eq:MinOp1c}, \\
 & \qquad \qquad \qquad   \dfrac{1}{m_i}+\mu_i ={\chi} \,,\quad \forall i\label{eq:dualprobred3}.
\end{align}%
\end{subequations}%
In this case, the robust optimal allocation profile tends to make the inertia allocations $m_i$ equal for all $i$, inducing a {\em valley-filling} strategy that allocates the entire inertia budget and prioritizes buses with lowest inertia first.

\section{Case study: 12-Bus-Three-Region System}
\label{Section: Case study}

\begin{figure}[t]
\centering
\includegraphics[width=3.4in]{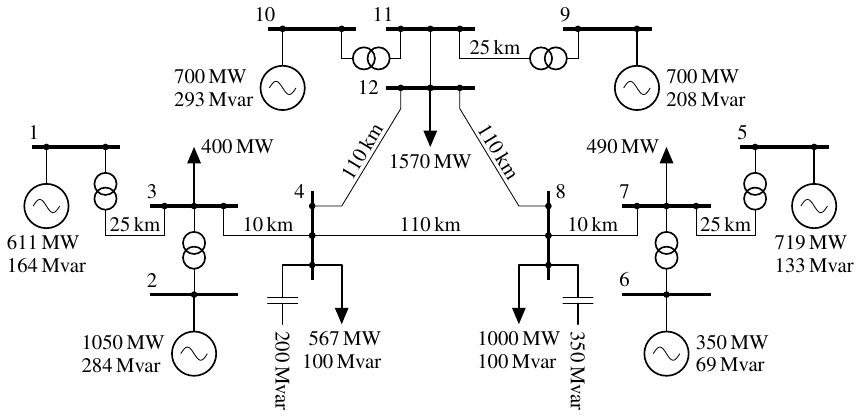}
\caption{A 12 bus three-region test case. Grid parameters are available in \cite{TSB-TL-DJH:15}.}
\label{fig_sim}
\end{figure}

In this section, we investigate a 12-bus case study illustrated in Figure~\ref{fig_sim}. The system parameters are based on a modified two-area system from \cite[Example 12.6]{PK:94} with an additional third area, as introduced in \cite{TSB-TL-DJH:15}. In this system, three buses are not available for inertia allocation, and are therefore eliminated via Kron reduction, resulting in a 9-bus equivalent.

We investigate this example computationally, using Algorithm~\ref{algorithm} to drive standard gradient-based optimization routines, while highlighting parallels to our analytic results.
We analyze different parametric scenarios and compare the inertia allocation and the performance of the proposed numerical optimization (a local optimum) with two plausible heuristics: the first one can be deduced from the conclusions in \cite{AU-TB-GA:14,TSB-TL-DJH:15}, and consists in allocating the available budget uniformly across the grid, in the absence of capacity constraints, that is, $m_{i} = {m}_{\text{uni}} = {m}_{\text{bdg}}/n$; the second follows from the intuition developed in Theorem \ref{theorem: Performance index for primary control cost}, and consists in allocating the maximum inertia allowed by the bus capacity, in the absence of a budget constraint, that is, $m_i = \overline{m_i}$ (which we set as $\overline{m_{i}}= 4 \underline{m_{i}}$).
We consider two disturbance scenarios:
a {\it uniform} disturbance affecting all nodes identically, and
a {\it localized} disturbance at node 4 alone.
For the performance metric, we choose $N=L$ and $S=I_9$.

\newlength{\barplotwidth}
\setlength{\barplotwidth}{3.4in}

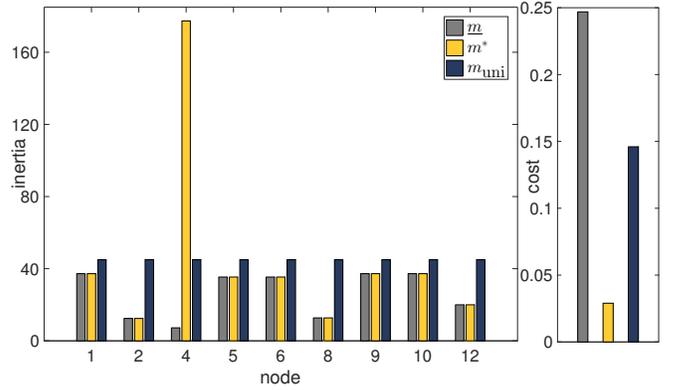
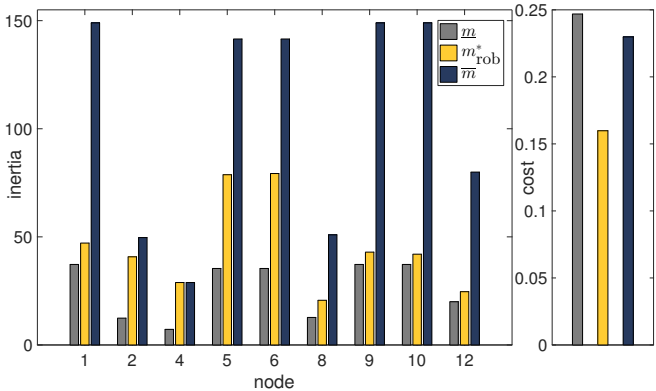
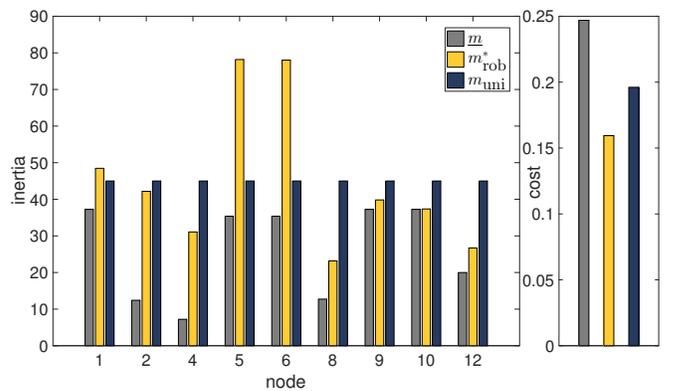
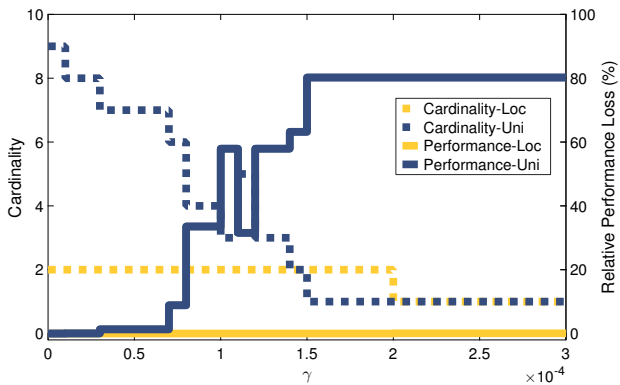
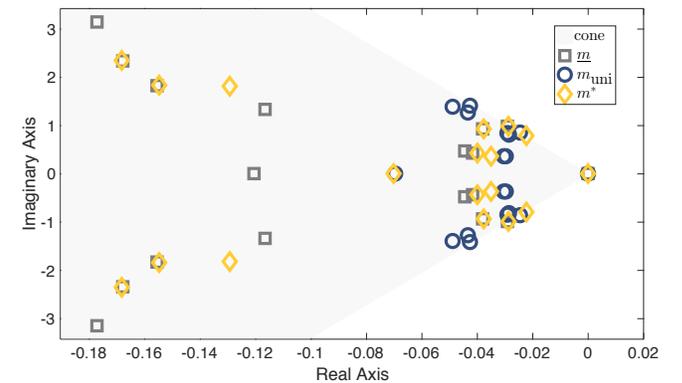
\begin{figure*}[htbp]
\begin{subfigure}[b]{0.49\textwidth}
%
%
\definecolor{mycolor1}{rgb}{0.49412,0.49412,0.49412}%
\definecolor{mycolor2}{HTML}{D59B2D}%
\definecolor{mycolor3}{HTML}{253F5B}%
\hspace*{-0.5em}
\begin{tikzpicture}

\begin{axis}[%
width=2.5in,
height=1.5in,
at={(-0.289in,4.95in)},
scale only axis,
bar shift auto,
xmin=0.511111111111111,
xmax=9.48888888888889,
xtick={1,2,3,4,5,6,7,8,9},
xticklabels={$1$, $2$, $4$, $5$, $6$, $8$, $9$, $10$, $12$},
xlabel style={font=\color{black}},
xlabel={\small{node}},
ymin=0,
ymax=155,
ytick={ 0,  50, 100, 150},
ymajorgrids,
ylabel style={font=\color{black}},
ylabel={\small{inertia}},
axis background/.style={fill=white},
yticklabel style = {font=\footnotesize,xshift=0ex},
xticklabel style = {font=\footnotesize,yshift=0ex},
no markers,
ylabel shift=-5pt,
every axis plot/.append style={ultra thin},
legend entries={$\underline{m}$, $m^*$, $\overline{m}$},
legend style={legend cell align=left, align=left, draw=none, font=\small},
legend style={at={(0.35,0.95)}, anchor=north east}
]

\addplot[ybar, bar width=4, fill=mycolor1, draw=black, area legend] table[row sep=crcr] {%
1	37.24225668\\
2	12.41408556\\
3	7.219268219\\
4	35.38014385\\
5	35.38014385\\
6	12.73239545\\
7	37.24225668\\
8	37.24225668\\
9	19.98986085\\
};
\addlegendentry{$\, \underline{m}$}

\addplot[ybar, bar width=4, fill=mycolor2, draw=black, area legend] table[row sep=crcr] {%
1	148.9506\\
2	33.7339\\
3	28.8696\\
4	141.4971\\
5	57.804\\
6	21.4995\\
7	147.1673\\
8	38.7739\\
9	24.8301\\
};
\addlegendentry{$\, m^*$}

\addplot[ybar, bar width=4, fill=mycolor3, draw=black, area legend] table[row sep=crcr] {%
1	148.9690267\\
2	49.65634224\\
3	28.87707287\\
4	141.5205754\\
5	141.5205754\\
6	50.92958179\\
7	148.9690267\\
8	148.9690267\\
9	79.95944341\\
};
\addlegendentry{$\, \overline{m}$}

\end{axis}

\begin{axis}[%
width=0.27in,
height=1.5in,
at={(2.58in, 4.95in)},
scale only axis,
bar shift auto,
xmin=0.4,
xmax=1.7,
xtick={\empty},
ymin=0,
ymax=0.165,
ytick={0, 0.05,  0.1, 0.15},
yticklabels={$0$, $0.05$,  $0.1$, $0.15$},
ylabel style={font=\color{black}},
ylabel={\small{cost}},
ymajorgrids,
axis background/.style={fill=white},
yticklabel style = {font=\footnotesize,xshift=0ex},
xticklabel style = {font=\footnotesize,yshift=0ex},
no markers,
ylabel shift=-9pt,
every axis plot/.append style={ultra thin},
legend style={legend cell align=left, align=left, draw=black}
]
\addplot[ybar, bar width=4, fill=mycolor1, draw=black, area legend] table[row sep=crcr] {%
1	0.1597\\
};

\addplot[ybar, bar width=4, fill=mycolor2, draw=black, area legend] table[row sep=crcr] {%
1.05	0.142\\
};

\addplot[ybar, bar width=4, fill=mycolor3, draw=black, area legend] table[row sep=crcr] {%
1.1	0.1527\\
};
\end{axis}
\end{tikzpicture}%
\caption{Optimal inertia allocation for a uniform disturbance subject to capacity constraints \eqref{eq:MinOp1c}.}
\label{fig:grid1} \par \medskip \vfill
\end{subfigure}%
\hfill \begin{subfigure}[b]{0.49\textwidth}
%
%
\definecolor{mycolor1}{rgb}{0.49412,0.49412,0.49412}%
\definecolor{mycolor2}{HTML}{D59B2D}%
\definecolor{mycolor3}{HTML}{253F5B}%
\hspace*{-0.5em}
\begin{tikzpicture}

\begin{axis}[%
width=2.5in,
height=1.5in,
at={(-0.289in,4.95in)},
scale only axis,
bar shift auto,
xmin=0.511111111111111,
xmax=9.48888888888889,
xtick={1,2,3,4,5,6,7,8,9},
xticklabels={$1$, $2$, $4$, $5$, $6$, $8$, $9$, $10$, $12$},
xlabel style={font=\color{black}},
xlabel={\small{node}},
ymin=0,
ymax=93,
ytick={ 0,  30, 60, 90},
ymajorgrids,
ylabel style={font=\color{black}},
ylabel={\small{inertia}},
axis background/.style={fill=white},
yticklabel style = {font=\footnotesize,xshift=0ex},
xticklabel style = {font=\footnotesize,yshift=0ex},
no markers,
ylabel shift=-5pt,
every axis plot/.append style={ultra thin},
legend entries={$\underline{m}$, $m^*$, $m_\text{uni}$},
legend style={legend cell align=left, align=left, draw=none, font=\small},
legend style={at={(0.35,0.95)}, anchor=north east}
]

\addplot[ybar, bar width=4, fill=mycolor1, draw=black, area legend] table[row sep=crcr] {%
1	37.24225668\\
2	12.41408556\\
3	7.219268219\\
4	35.38014385\\
5	35.38014385\\
6	12.73239545\\
7	37.24225668\\
8	37.24225668\\
9	19.98986085\\
};
\addlegendentry{$\, \underline{m}$}

\addplot[ybar, bar width=4, fill=mycolor2, draw=black, area legend] table[row sep=crcr] {%
1	85.375\\
2	23.63\\
3	26.423\\
4	80.569\\
5	42.668\\
6	18.549\\
7	70.457\\
8	37.242\\
9	19.991\\
};
\addlegendentry{$\, m^*$}

\addplot[ybar, bar width=4, fill=mycolor3, draw=black, area legend] table[row sep=crcr] {%
1	45\\
2	45\\
3	45\\
4	45\\
5	45\\
6	45\\
7	45\\
8	45\\
9	45\\
};
\addlegendentry{$\, \overline{m}$}

\end{axis}

\begin{axis}[%
width=0.27in,
height=1.5in,
at={(2.58in, 4.95in)},
scale only axis,
bar shift auto,
xmin=0.4,
xmax=1.7,
xtick={\empty},
ymin=0,
ymax=0.165,
ytick={0, 0.05,  0.1, 0.15},
yticklabels={$0$, $0.05$,  $0.1$, $0.15$},
ylabel style={font=\color{white!15!black}},
ylabel={\small{cost}},
ymajorgrids,
yticklabel style = {font=\footnotesize,xshift=0ex},
xticklabel style = {font=\footnotesize,yshift=0ex},
axis background/.style={fill=white},
no markers,
ylabel shift=-9pt,
every axis plot/.append style={ultra thin},
legend style={legend cell align=left, align=left, draw=black}
]
\addplot[ybar, bar width=4, fill=mycolor1, draw=black, area legend] table[row sep=crcr] {%
1	0.1597\\
};

\addplot[ybar, bar width=4, fill=mycolor2, draw=black, area legend] table[row sep=crcr] {%
1.05	0.1461\\
};

\addplot[ybar, bar width=4, fill=mycolor3, draw=black, area legend] table[row sep=crcr] {%
1.1	0.1566\\
};
\end{axis}
\end{tikzpicture}%
\caption{Optimal inertia allocation for a uniform disturbance subject to budget constraint \eqref{eq:MinOp1b}.}
\label{fig:grid2} \par \medskip \vfill
\end{subfigure}

\begin{subfigure}[t]{0.49\textwidth}
%
%
\definecolor{mycolor1}{rgb}{0.49412,0.49412,0.49412}%
\definecolor{mycolor2}{HTML}{D59B2D}%
\definecolor{mycolor3}{HTML}{253F5B}%
\hspace*{-0.5em}
\begin{tikzpicture}

\begin{axis}[%
width=2.5in,
height=1.5in,
at={(-0.289in,4.95in)},
scale only axis,
bar shift auto,
xmin=0.511111111111111,
xmax=9.48888888888889,
xtick={1,2,3,4,5,6,7,8,9},
xticklabels={$1$, $2$, $4$, $5$, $6$, $8$, $9$, $10$, $12$},
xlabel style={font=\color{black}},
xlabel={\small{node}},
ymin=0,
ymax=155,
ytick={ 0,  50, 100, 150},
ylabel style={font=\color{black}},
ylabel={\small{inertia}},
ymajorgrids,
ylabel shift=-5pt,
axis background/.style={fill=white},
yticklabel style = {font=\footnotesize,xshift=0ex},
xticklabel style = {font=\footnotesize,yshift=0ex},
no markers,
every axis plot/.append style={ultra thin},
legend entries={$\underline{m}$, $m^*$, $\overline{m}$},
legend style={legend cell align=left, align=left, draw=none, font=\small},
legend style={at={(0.35,0.95)}, anchor=north east}
]

\addplot[ybar, bar width=4, fill=mycolor1, draw=black, area legend] table[row sep=crcr] {%
1	37.24225668\\
2	12.41408556\\
3	7.219268219\\
4	35.38014385\\
5	35.38014385\\
6	12.73239545\\
7	37.24225668\\
8	37.24225668\\
9	19.98986085\\
};
\addlegendentry{$\, \underline{m}$}

\addplot[ybar, bar width=4, fill=mycolor2, draw=black, area legend] table[row sep=crcr] {%
1	37.244\\
2	12.414\\
3	28.877\\
4	35.383\\
5	38.4\\
6	12.736\\
7	37.245\\
8	37.246\\
9	19.992\\
};
\addlegendentry{$\, m^*$}

\addplot[ybar, bar width=4, fill=mycolor3, draw=black, area legend] table[row sep=crcr] {%
1	148.9690267\\
2	49.65634224\\
3	28.87707287\\
4	141.5205754\\
5	141.5205754\\
6	50.92958179\\
7	148.9690267\\
8	148.9690267\\
9	79.95944341\\
};
\addlegendentry{$\, \overline{m}$}
\end{axis}

\begin{axis}[%
width=0.27in,
height=1.5in,
at={(2.58in, 4.95in)},
scale only axis,
bar shift auto,
xmin=0.4,
xmax=1.7,
xtick={\empty},
ymin=0,
ymax=0.28,
ytick={0, 0.09,  0.18, 0.27},
yticklabels={$0$, $0.09$,  $0.18$, $0.27$},
ylabel style={font=\color{black}},
ylabel={\small{cost}},
ymajorgrids,
ylabel shift=-9pt,
yticklabel style = {font=\footnotesize,xshift=0ex},
xticklabel style = {font=\footnotesize,yshift=0ex},
no markers,
every axis plot/.append style={ultra thin},
axis background/.style={fill=white}
]
\addplot[ybar, bar width=4, fill=mycolor1, draw=black, area legend] table[row sep=crcr] {%
1	0.2469\\
};

\addplot[ybar, bar width=4, fill=mycolor2, draw=black, area legend] table[row sep=crcr] {%
1.05	0.1269\\
};

\addplot[ybar, bar width=4, fill=mycolor3, draw=black, area legend] table[row sep=crcr] {%
1.1	0.2298\\
};
\end{axis}

\end{tikzpicture}%
\caption{Optimal inertia allocation for a localized disturbance at node 4 subject to capacity constraints \eqref{eq:MinOp1c}.}
\label{fig:grid3} \par \medskip \vfill
\end{subfigure}%
\hfill \begin{subfigure}[t]{0.49\textwidth}
%
%
\definecolor{mycolor1}{rgb}{0.49412,0.49412,0.49412}%
\definecolor{mycolor2}{HTML}{D59B2D}%
\definecolor{mycolor3}{HTML}{253F5B}%
\hspace*{-0.5em}
\begin{tikzpicture}

\begin{axis}[%
width=2.5in,
height=1.5in,
at={(-0.289in,4.95in)},
scale only axis,
bar shift auto,
xmin=0.511111111111111,
xmax=9.48888888888889,
xtick={1,2,3,4,5,6,7,8,9},
xticklabels={$1$, $2$, $4$, $5$, $6$, $8$, $9$, $10$, $12$},
xlabel style={font=\color{black}},
xlabel={\small{node}},
ymin=0,
ymax=186,
ytick={  0,  45,  90, 135, 180},
ylabel style={font=\color{black}},
ylabel={\small{inertia}},
ymajorgrids,
axis background/.style={fill=white},
yticklabel style = {font=\footnotesize,xshift=0ex},
xticklabel style = {font=\footnotesize,yshift=0ex},
no markers,
ylabel shift=-5pt,
every axis plot/.append style={ultra thin},
legend entries={$\underline{m}$, $m^*$, $m_\text{uni}$},
legend style={legend cell align=left, align=left, draw=none, font=\small},
legend style={at={(0.26,0.95)}, anchor=north east}
]

\addplot[ybar, bar width=4, fill=mycolor1, draw=black, area legend] table[row sep=crcr] {%
1	37.24225668\\
2	12.41408556\\
3	7.219268219\\
4	35.38014385\\
5	35.38014385\\
6	12.73239545\\
7	37.24225668\\
8	37.24225668\\
9	19.98986085\\
};
\addlegendentry{$\, \underline{m}$}

\addplot[ybar, bar width=4, fill=mycolor2, draw=black, area legend] table[row sep=crcr] {%
1	37.2424\\
2	12.4145\\
3	177.3716\\
4	35.3804\\
5	35.3804\\
6	12.7326\\
7	37.2425\\
8	37.2425\\
9	19.9901\\
};
\addlegendentry{$\, m^*$}

\addplot[ybar, bar width=4, fill=mycolor3, draw=black, area legend] table[row sep=crcr] {%
1	45\\
2	45\\
3	45\\
4	45\\
5	45\\
6	45\\
7	45\\
8	45\\
9	45\\
};
\addlegendentry{$\, m_{\textup{uni}}$}

\end{axis}

\begin{axis}[%
width=0.27in,
height=1.5in,
at={(2.58in, 4.95in)},
scale only axis,
bar shift auto,
xmin=0.4,
xmax=1.7,
xtick={\empty},
ymin=0,
ymax=0.28,
ytick={0, 0.09,  0.18, 0.27},
yticklabels={$0$, $0.09$,  $0.18$, $0.27$},
ylabel style={font=\color{black}},
ylabel={\small{cost}},
axis background/.style={fill=white},
ymajorgrids,
yticklabel style = {font=\footnotesize,xshift=0ex},
xticklabel style = {font=\footnotesize,yshift=0ex},
no markers,
ylabel shift=-9pt,
every axis plot/.append style={ultra thin},
legend style={legend cell align=left, align=left, draw=black}
]
\addplot[ybar, bar width=4, fill=mycolor1, draw=black, area legend] table[row sep=crcr] {%
1	0.2469\\
};

\addplot[ybar, bar width=4, fill=mycolor2, draw=black, area legend] table[row sep=crcr] {%
1.05	0.029\\
};

\addplot[ybar, bar width=4, fill=mycolor3, draw=black, area legend] table[row sep=crcr] {%
1.1	0.146\\
};
\end{axis}
\end{tikzpicture}%
\caption{Optimal inertia allocation for a localized disturbance at node 4 subject to budget constraint \eqref{eq:MinOp1b}.}
\label{fig:grid4}
\end{subfigure}

\begin{subfigure}[t]{0.49\textwidth}
%
%
\definecolor{mycolor1}{rgb}{0.49412,0.49412,0.49412}%
\definecolor{mycolor2}{HTML}{D59B2D}%
\definecolor{mycolor3}{HTML}{253F5B}%
\hspace*{-0.5em}
\begin{tikzpicture}

\begin{axis}[%
width=2.5in,
height=1.5in,
at={(-0.289in,4.95in)},
scale only axis,
bar shift auto,
xmin=0.511111111111111,
xmax=9.48888888888889,
xtick={1,2,3,4,5,6,7,8,9},
xticklabels={$1$, $2$, $4$, $5$, $6$, $8$, $9$, $10$, $12$},
xlabel style={font=\color{black}},
xlabel={\small{node}},
ymin=0,
ymax=155,
ytick={ 0,  50, 100, 150},
ylabel style={font=\color{black}},
ylabel={\small{inertia}},
ymajorgrids,
axis background/.style={fill=white},
yticklabel style = {font=\footnotesize,xshift=0ex},
xticklabel style = {font=\footnotesize,yshift=0ex},
no markers,
ylabel shift=-5pt,
every axis plot/.append style={ultra thin},
legend entries={$\underline{m}$, $m_\text{rob}$, $\overline{m}$},
legend style={legend cell align=left, align=left, draw=none, font=\small},
legend style={at={(0.4,0.95)}, anchor=north east}
]

\addplot[ybar, bar width=4, fill=mycolor1, draw=black, area legend] table[row sep=crcr] {%
1	37.24225668\\
2	12.41408556\\
3	7.219268219\\
4	35.38014385\\
5	35.38014385\\
6	12.73239545\\
7	37.24225668\\
8	37.24225668\\
9	19.98986085\\
};
\addlegendentry{$\, \underline{m}$}

\addplot[ybar, bar width=4, fill=mycolor2, draw=black, area legend] table[row sep=crcr] {%
1	47.1339147534208\\
2	40.7726904247314\\
3	28.877035786861\\
4	78.7379132202302\\
5	79.2944091780122\\
6	20.6699475213137\\
7	42.9442918788029\\
8	41.9911873533226\\
9	24.6374293768444\\
};
\addlegendentry{$\, m_{\textup{rob}}^*$}

\addplot[ybar, bar width=4, fill=mycolor3, draw=black, area legend] table[row sep=crcr] {%
1	148.9690267\\
2	49.65634224\\
3	28.87707287\\
4	141.5205754\\
5	141.5205754\\
6	50.92958179\\
7	148.9690267\\
8	148.9690267\\
9	79.95944341\\
};
\addlegendentry{$\, \overline{m}$}

\end{axis}

\begin{axis}[%
width=0.27in,
height=1.5in,
at={(2.58in, 4.95in)},
scale only axis,
bar shift auto,
xmin=0.4,
xmax=1.7,
xtick={\empty},
ymin=0,
ymax=0.28,
ytick={0, 0.09,  0.18, 0.27},
yticklabels={$0$, $0.09$,  $0.18$, $0.27$},
ylabel style={font=\color{black}},
ylabel={\small{cost}},
ymajorgrids,
axis background/.style={fill=white},
yticklabel style = {font=\footnotesize,xshift=0ex},
xticklabel style = {font=\footnotesize,yshift=0ex},
no markers,
ylabel shift=-9pt,
every axis plot/.append style={ultra thin},
legend style={legend cell align=left, align=left, draw=black}
]
\addplot[ybar, bar width=4, fill=mycolor1, draw=black, area legend] table[row sep=crcr] {%
1	0.2469\\
};

\addplot[ybar, bar width=4, fill=mycolor2, draw=black, area legend] table[row sep=crcr] {%
1.05	0.1598\\
};

\addplot[ybar, bar width=4, fill=mycolor3, draw=black, area legend] table[row sep=crcr] {%
1.1	0.2298\\
};

\end{axis}


\end{tikzpicture}%
\caption{Robust inertia allocation subject to capacity constraints \eqref{eq:MinOp1c}.}
\label{fig:robust1}
\end{subfigure}%
\hfill \begin{subfigure}[t]{0.49\textwidth}
%
%
\definecolor{mycolor1}{rgb}{0.49412,0.49412,0.49412}%
\definecolor{mycolor2}{HTML}{D59B2D}%
\definecolor{mycolor3}{HTML}{253F5B}%
\hspace*{-0.5em}
\begin{tikzpicture}

\begin{axis}[%
width=2.5in,
height=1.5in,
at={(-0.289in,4.95in)},
scale only axis,
bar shift auto,
xmin=0.511111111111111,
xmax=9.48888888888889,
xtick={1,2,3,4,5,6,7,8,9},
xticklabels={$1$, $2$, $4$, $5$, $6$, $8$, $9$, $10$, $12$},
xlabel style={font=\color{black}},
xlabel={\small{node}},
ymin=0,
ymax=93,
ytick={ 0,  30,  60, 90},
yticklabels={ $0$, $30$, $60$, $90$},
ylabel style={font=\color{black}},
ylabel={\small{inertia}},
ymajorgrids,
ylabel shift=-5pt,
axis background/.style={fill=white},
yticklabel style = {font=\footnotesize,xshift=0ex},
xticklabel style = {font=\footnotesize,yshift=0ex},
no markers,
every axis plot/.append style={ultra thin},
legend entries={$\underline{m}$, $m_\text{rob}$, $m_\text{uni}$},
legend style={legend cell align=left, align=left, draw=none, font=\small},
legend style={at={(0.35,0.95)}, anchor=north east}
]

\addplot[ybar, bar width=4, fill=mycolor1, draw=black, area legend] table[row sep=crcr] {%
1	37.24225668\\
2	12.41408556\\
3	7.219268219\\
4	35.38014385\\
5	35.38014385\\
6	12.73239545\\
7	37.24225668\\
8	37.24225668\\
9	19.98986085\\
};
\addlegendentry{$\, \underline{m}$}

\addplot[ybar, bar width=4, fill=mycolor2, draw=black, area legend] table[row sep=crcr] {%
1	48.465\\
2	42.172\\
3	31.08\\
4	78.18\\
5	78.021\\
6	23.173\\
7	39.826\\
8	37.364\\
9	26.705\\
};
\addlegendentry{$\, m_\text{rob}^*$}

\addplot[ybar, bar width=4, fill=mycolor3, draw=black, area legend] table[row sep=crcr] {%
1	45\\
2	45\\
3	45\\
4	45\\
5	45\\
6	45\\
7	45\\
8	45\\
9	45\\
};
\addlegendentry{$\, m_{\textup{uni}}$}

\end{axis}

\begin{axis}[%
width=0.27in,
height=1.5in,
at={(2.58in, 4.95in)},
scale only axis,
bar shift auto,
xmin=0.4,
xmax=1.7,
xtick={\empty},
ymin=0,
ymax=0.28,
ytick={0, 0.09,  0.18, 0.27},
yticklabels={$0$, $0.09$,  $0.18$, $0.27$},
ylabel style={font=\color{black}},
ylabel={\small{cost}},
ymajorgrids,
ylabel shift=-9pt,
axis background/.style={fill=white},
yticklabel style = {font=\footnotesize,xshift=0ex},
xticklabel style = {font=\footnotesize,yshift=0ex},
no markers,
every axis plot/.append style={ultra thin},
legend style={legend cell align=left, align=left, draw=black}
]
\addplot[ybar, bar width=4, fill=mycolor1, draw=black, area legend] table[row sep=crcr] {%
1	0.2469\\
};

\addplot[ybar, bar width=4, fill=mycolor2, draw=black, area legend] table[row sep=crcr] {%
1.05	0.1594\\
};

\addplot[ybar, bar width=4, fill=mycolor3, draw=black, area legend] table[row sep=crcr] {%
1.1	0.1961\\
};
\end{axis}

\end{tikzpicture}%
\caption{Robust inertia allocation subject to budget constraint \eqref{eq:MinOp1b}.}
\label{fig:robust2}
\end{subfigure}

\begin{subfigure}[t]{0.49\textwidth}
\centering
%

\definecolor{mycolor2}{HTML}{253F5B}%
\definecolor{mycolor1}{HTML}{D59B2D}%
\begin{tikzpicture}

\begin{axis}[%
width=2.58in,
height=1.59in,
at={(1.17in,1.077in)},
scale only axis,
xmin=0,
xmax=0.0003,
xmajorticks=false,
every outer y axis line/.append style={black},
every y tick label/.append style={font=\color{black}},
every y tick/.append style={black},
ymin=-1,
ymax=11,
ytick={ 0,  2,  4,  6,  8, 10},
ylabel={\small{Cardinality}},
axis background/.style={fill=white},
yticklabel style = {font=\footnotesize,xshift=0ex},
xticklabel style = {font=\footnotesize,yshift=0ex},
]

\addplot[const plot, color=mycolor1, dotted, line width=2.5pt, forget plot] table[row sep=crcr] {%
0	2\\
0.000199999999999978	1\\
0.000399999999999956	1\\
};

\addplot[const plot, color=mycolor2, dotted, line width=2.5pt, forget plot] table[row sep=crcr] {%
0	9\\
9.99999999962142e-06	8\\
3.00000000006406e-05	7\\
6.99999999991263e-05	6\\
8.00000000005241e-05	4\\
9.99999999997669e-05	3\\
0.000109999999999388	5\\
0.000120000000000786	3\\
0.000140000000000029	2\\
0.00014999999999965	1\\
0.000310000000000699	1\\
};

\end{axis}

\begin{axis}[%
width=2.58in,
height=1.59in,
at={(1.17in,1.077in)},
scale only axis,
xmin=0,
xmax=0.0003,
every y tick label/.append style={font=\color{black}},
every y tick/.append style={black},
ymin=-10,
ymax=110,
ytick={0,  20,  40,  60,  80, 100},
ylabel={\small{Relative Performance Loss [\%]}},
xtick={ 0, 0.5e-4, 1e-4, 1.5e-4, 2e-4, 2.5e-4, 3e-4},
xticklabels={ $0$, $0.5$, $1$, $1.5$, $2$, $2.5$, $3$},
xmajorgrids,
xlabel={\small{$\gamma\, (10^{-4})$}},
ymajorgrids,
yticklabel style = {font=\footnotesize,xshift=0ex},
xticklabel style = {font=\footnotesize,yshift=0ex},
ylabel near ticks, yticklabel pos=right,
legend entries={$\text{localized, card}$, $\text{uniform, card}$, $\text{localized, perf}$, $\text{uniform, perf}$},
legend style={legend cell align=left, align=left, draw=none, font=\small},
legend style={at={(1,0.27)},anchor=south east}
]

\addlegendimage{no markers, mycolor1, dotted, ultra thick}
\addlegendimage{no markers, mycolor2, dotted, ultra thick}
\addlegendimage{no markers, mycolor1, ultra thick}
\addlegendimage{no markers, mycolor2, ultra thick}

\addplot[const plot, color=mycolor1, line width=2.5pt] table[row sep=crcr] {%
0	0\\
0.000200000000000006	0.0734773241000624\\
0.000399999999999998	0.0734773241000624\\
};

\addplot[const plot, color=mycolor2, line width=2.5pt] table[row sep=crcr] {%
0	0\\
1.00000000031741e-05	0.0338174009277736\\
2.99999999953116e-05	1.3147444386012\\
6.99999999937972e-05	8.85228336707502\\
7.99999999969714e-05	33.540299195011\\
0.00010000000000332	57.9217266636287\\
0.000110000000006494	31.5530121701586\\
0.000119999999995457	57.9217266636287\\
0.000140000000001805	63.2090665549401\\
0.000150000000004979	80.2189677908254\\
0.000309999999998922	80.2189677908254\\
};
\end{axis}

\end{tikzpicture}%
\caption{\small{Relative performance loss (\%) as a function of penalty $\gamma$ with capacity constraints. 0\%, 100\% correspond to the optimal allocation, no additional allocation respectively.}}
\label{fig:gridN}
\end{subfigure}%
\hfill \begin{subfigure}[t]{0.49\textwidth}
\centering
%
%
\definecolor{mycolor1}{rgb}{0.955,0.955,0.955}%
\definecolor{mycolor2}{rgb}{0.49412,0.49412,0.49412}%
\definecolor{mycolor3}{HTML}{253F5B}
\definecolor{mycolor4}{HTML}{D59B2D}
\hspace*{-2em}
\begin{tikzpicture}

\begin{axis}[%
width=2.58in,
height=1.59in,
at={(1.17in,1.077in)},
scale only axis,
xmin=-0.1905,
xmax=0.015,
xlabel style={font=\color{black}},
xlabel={\small{Real Axis}},
ymin=-3.425,
ymax=3.425,
ytick={-3, -1.5, 0, 1.5, 3},
xtick={-0.18, -0.15, -0.12, -0.09, -0.06, -0.03, 0},
ylabel style={font=\color{black}},
ylabel={\small{Imaginary Axis}},
axis background/.style={fill=white},
xmajorgrids,
ymajorgrids,
yticklabel style = {font=\footnotesize,xshift=0ex},
xticklabel style = {font=\footnotesize,yshift=0ex},
legend entries={$\textit{cone}$, $\underline{m}$, $m_{\textup{uni}}$, $m^*$},
legend style={legend cell align=left, align=left, draw=none, font=\small},
legend style={at={(0.3,0.7)}, anchor=north east}
]

\addplot[area legend, line width=1.5pt, draw=mycolor1, fill=mycolor1]
table[row sep=crcr] {%
x	y\\
-0.19	-3.4\\
-0.1003	-3.4\\
0	0\\
-0.1003	3.4\\
-0.19	3.4\\
}--cycle;
\addlegendentry{cone}

\addplot [color=mycolor2, line width=1.5pt, draw=none, mark size=1.5pt, mark=square, mark options={solid, mycolor2}]
  table[row sep=crcr]{%
-0.177346386981727	3.14877210542865\\
-0.177346386981727	-3.14877210542865\\
-0.167876787481245	2.33789771430358\\
-0.167876787481245	-2.33789771430358\\
-0.155567017804815	1.82944164651823\\
-0.155567017804815	-1.82944164651823\\
-0.116581901167269	1.33283558374031\\
-0.116581901167269	-1.33283558374031\\
-0.0291860962831962	0.985141027472319\\
-0.0291860962831962	-0.985141027472319\\
-0.0381312233097071	0.930923193933396\\
-0.0381312233097071	-0.930923193933396\\
-0.0445352858043671	0.470763592612919\\
-0.0445352858043671	-0.470763592612919\\
-0.0416297939453227	0.432996864507253\\
-0.0416297939453227	-0.432996864507253\\
-0.120622191727754	0\\
-0	0\\
};
\addlegendentry{$\, \underline{m}$}

\addplot [color=mycolor3, line width=1.5pt, draw=none, mark size=1.5pt, mark=o, mark options={solid, mycolor3}]
  table[row sep=crcr]{%
-0.0425987224622506	1.41358773766826\\
-0.0425987224622506	-1.41358773766826\\
-0.0489854002347561	1.38914257110187\\
-0.0489854002347561	-1.38914257110187\\
-0.0435092807422623	1.27418368087978\\
-0.0435092807422623	-1.27418368087978\\
-0.0286748638845638	0.814250404778329\\
-0.0286748638845638	-0.814250404778329\\
-0.0246761094701566	0.854190995741252\\
-0.0246761094701566	-0.854190995741252\\
-0.0291704845971983	0.845725749028127\\
-0.0291704845971983	-0.845725749028127\\
-0.0304641344547854	0.36291166098301\\
-0.0304641344547854	-0.36291166098301\\
-0.0296431850430912	0.369562972478171\\
-0.0296431850430912	-0.369562972478171\\
-4.44089209850063e-16	0\\
-0.0695756053126537	0\\
};
\addlegendentry{$m_{\textup{uni}}$}

\addplot [color=mycolor4, line width=1.5pt, draw=none, mark size=1.5pt, mark=diamond, mark options={solid, mycolor4}]
  table[row sep=crcr]{%
-0.168339759383079	2.34928435294055\\
-0.168339759383079	-2.34928435294055\\
-0.129408706665423	1.81369560948562\\
-0.129408706665423	-1.81369560948562\\
-0.154793777408708	1.83434584316143\\
-0.154793777408708	-1.83434584316143\\
-0.0287430867900644	0.986777815648268\\
-0.0287430867900644	-0.986777815648268\\
-0.0376784322504808	0.934736626686396\\
-0.0376784322504808	-0.934736626686396\\
-0.0222215099851133	0.787836676392018\\
-0.0222215099851133	-0.787836676392018\\
-0.0400238173224161	0.434771378331579\\
-0.0400238173224161	-0.434771378331579\\
-0.0349770446932967	0.360314776773839\\
-0.0349770446932967	-0.360314776773839\\
1.33226762955019e-15	0\\
-0.0701998863917703	0\\
};
\addlegendentry{$m^*$}

\end{axis}
\end{tikzpicture}%
\caption{\small{The eigenvalue spectrum of the state matrix $A$ for different inertia profiles, where $m^{*}$ has been optimized for a localized disturbance at node 4.}}
\label{fig:grid8}
\end{subfigure}
\caption{Optimal inertia allocations, performance comparison for the test case in Figure~\ref{fig_sim}, under different scenarios.}
\label{fig: allocations}
\end{figure*}


\begin{figure*}[!tb]
  \centering
      \input{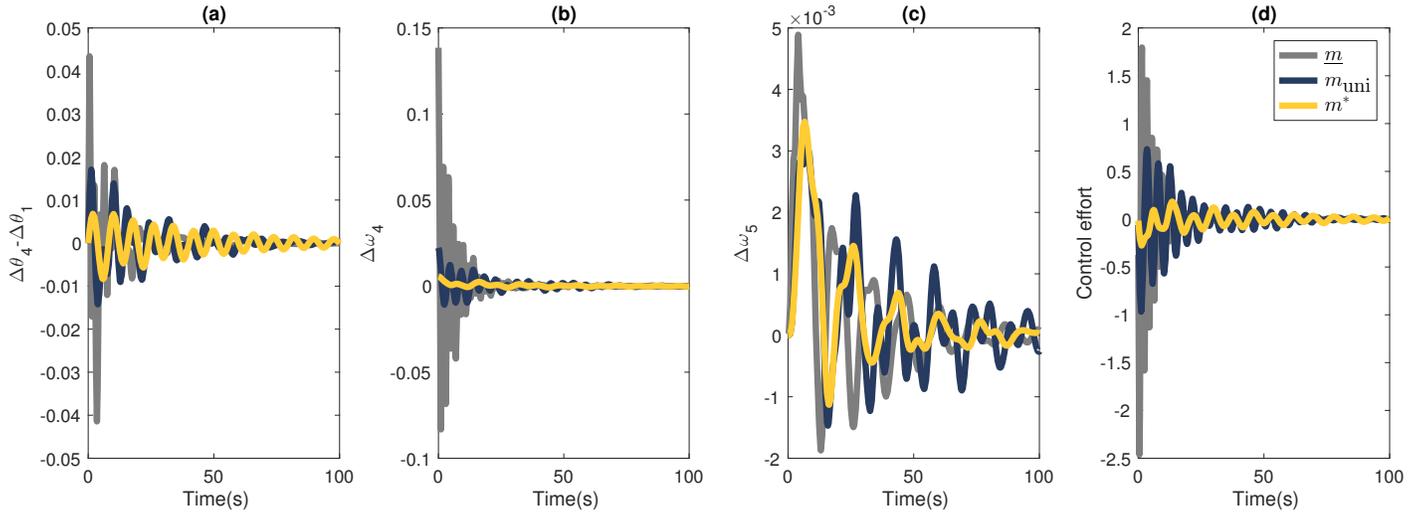}
\caption{Time-domain plots for angle differences, frequencies, and control effort $m_4 \ddot \theta_4$ for a localized disturbance at node 4.} 
\label{fig:grid7}
\end{figure*}

We draw the following conclusions from  the above test cases {-- some of which are perhaps surprising and counterintuitive.}
\begin{enumerate}[label=(\alph*), leftmargin=0.1cm, itemindent=0.5cm]
\setlength{\itemsep}{2pt}

\item Our locally optimal solution achieves the best performance among the different heuristics in all scenarios; see Figure~\ref{fig: allocations}. 

\item For uniform disturbances with capacity constraints on the individual buses, the optimal solution does not correspond to allocating the maximum possible inertia at every bus (Figure~\ref{fig:grid1}). When only a total budget constraint is present, the optimal solution is remarkably different from the uniform allocation of inertia at the different nodes (Figure~\ref{fig:grid2}). For both scenarios, the performance improvement with respect to the initial allocation and the different heuristics is modest, and confirms the intuition developed for the {two-area} case (Section~\ref{subsec: two-area}) regarding the flatness of the cost function.

\item In stark contrast, a localized disturbance results in adding inertia dominantly to the disturbed node (node 4) as an optimum choice. The latter is also in line with the results presented for the two-area case and the closed-form results in Theorem~\ref{theorem:  Performance index for primary control cost}. Furthermore, an additional inertia allocation to all other (undisturbed) nodes may be detrimental for the performance, as shown in Figure~\ref{fig:grid3}.

\item The robust allocation approach proposed in Section~\ref{Subsection: Robustness} is investigated in Figures~\ref{fig:robust1} and \ref{fig:robust2}. The ensuing inertia profiles in addition to being robust to disturbance location, also result in a significantly lower worst-case cost compared to the heuristic allocations. We also observe a generally flatter optimal inertia profile, in tune with the conclusions following the re-formulation \eqref{eq:genminprobfb-3} developed for another cost function.

\item The sparsity-promoting approach proposed in Section~\ref{Subsection: l1_norm} is examined in Figure~\ref{fig:gridN}. For a uniform disturbance without a sparsity penalty, inertia is allocated at all nine buses of the network. A modest penalty of $\gamma =6e{-5}$, however, yields an allocation at only seven buses with a mere 1.3\% degradation in performance. For sparser allocations, the performance loss becomes more relevant.
The optimal allocation is inherently sparser in the case of localized disturbances. Even without a sparsity penalty, virtual inertia would be assigned to only buses 4 and 6, when the disturbance is at node 4. An allocation exclusively at bus 4 (economically preferable), with negligible performance loss can be arrived at, with a penalty of $\gamma > 2e{-4}$.

\item Figure~\ref{fig:grid7} shows the time-domain responses to a localized impulse at node 4, modeling a post-fault condition. Subfigure (a) (respectively, (b)) 
shows that the optimal inertia allocation according to the proposed $\mathscr H_2$ performance criteria is also superior in terms of frequency overshoot and angle differences (respectively, frequencies). Subfigure (c) displays the frequency response at node 5 of the system. Note from the scale of this plot that the deviations are insignificant. Similar comments also apply to all other signals which are not displayed here. Finally, Subfigure (d) shows the control effort $m_4 \cdot \ddot{\theta}_{4}$ expended by the virtual inertia emulation at the disturbed bus. Perhaps surprisingly, we observe that the optimal allocation $m=m^*$ requires the least control effort.

\item Figure~\ref{fig:grid8} plots the eigenvalue spectrum for  different inertia profiles. 
The initial inertia profile $\underline m$, marginally outperforms other allocations with respect to both the best damping asymptote (most damped nonzero eigenvalues) as well as the best damping ratio (narrowest cone). As is apparent from the plots in Figure~\ref{fig:grid7}, this case also leads to inferior time-domain performance compared to the optimal allocation $m^{*}$, which has slightly poorer damping asymptote and ratio. These observations reveal that the spectrum holds only partial information, and advocate the use of the $\mathscr H_{2}$ norm as opposed to spectral performance metrics (as in \cite{TSB-TL-DJH:15}).

\end{enumerate}

\section{Conclusions}
\label{Section: Conclusions}

We considered the problem of placing virtual inertia in power grids based on an $\mathscr{H}_2$ norm performance metric reflecting network coherency. This formulation gave rise to a large-scale and non-convex optimization program. For certain cost functions, problem instances, and in the low-dimensional {two-area} case, we could derive closed-form solutions yielding some, possibly surprising insights. Next, we developed a computational approach based on an explicit gradient formulation and validated our results on a three-region network. Suitable time-domain simulations demonstrated the efficacy of our locally optimal inertia allocations over intuitive heuristics.
We also examined the  problem of allocating a finite number of virtual inertia units via a sparsity-promoting regularization.\\ 
Our computational and analytic results are well aligned and suggest insightful strategies for the optimal allocation of virtual inertia. Contrary to popular belief, it is the location of disturbance and the placement of inertia in the grid, rather than the total inertia in a power system that dictates its resilience.
We envision that these results will find application in stabilizing low-inertia grids through strategically placed virtual inertia units. As part of our future work, we consider the extension to more detailed system models and specifications as well as a comparison with the results in \cite{TSB-TL-DJH:15}.

\section*{Acknowledgements}
 The authors wish to thank Mihailo Jovanovic, Andreas Ulbig, Theodor Borsche, Dominic Gro\ss{}, and Ulrich M\"unz for their comments on the problem setup and analysis methods.
\bibliographystyle{IEEEtran}

\begin{thebibliography}{10}
\providecommand{\url}[1]{#1}
\csname url@samestyle\endcsname
\providecommand{\newblock}{\relax}
\providecommand{\bibinfo}[2]{#2}
\providecommand{\BIBentrySTDinterwordspacing}{\spaceskip=0pt\relax}
\providecommand{\BIBentryALTinterwordstretchfactor}{4}
\providecommand{\BIBentryALTinterwordspacing}{\spaceskip=\fontdimen2\font plus
\BIBentryALTinterwordstretchfactor\fontdimen3\font minus
  \fontdimen4\font\relax}
\providecommand{\BIBforeignlanguage}[2]{{%
\expandafter\ifx\csname l@#1\endcsname\relax
\typeout{** WARNING: IEEEtran.bst: No hyphenation pattern has been}%
\typeout{** loaded for the language `#1'. Using the pattern for}%
\typeout{** the default language instead.}%
\else
\language=\csname l@#1\endcsname
\fi
#2}}
\providecommand{\BIBdecl}{\relax}
\BIBdecl

\bibitem{WW-KE-GB-KJ:15}
W.~Winter, K.~Elkington, G.~Bareux, and J.~Kostevc, ``Pushing the limits:
  Europe's new grid: Innovative tools to combat transmission bottlenecks and
  reduced inertia,'' \emph{IEEE Power and Energy Magazine}, vol.~13, no.~1, Jan
  2015.

\bibitem{MM-BF-BK-MS:15}
M.~Milligan, B.~Frew, B.~Kirby, M.~Schuerger, K.~Clark, D.~Lew, P.~Denholm,
  B.~Zavadil, M.~O'Malley, and B.~Tsuchida, ``Alternatives no more: Wind and
  solar power are mainstays of a clean, reliable, affordable grid,'' \emph{IEEE
  Power and Energy Magazine}, vol.~13, no.~6, 2015.

\bibitem{AU-TB-GA:14}
A.~Ulbig, T.~S. Borsche, and G.~Andersson, ``Impact of low rotational inertia
  on power system stability and operation,'' in \emph{IFAC World Congress},
  2014.

\bibitem{SN-SD-MCC:13}
N.~Soni, S.~Doolla, and M.~C. Chandorkar, ``Improvement of transient response
  in microgrids using virtual inertia,'' \emph{IEEE Transactions on Power
  Delivery}, vol.~28, no.~3, 2013.

\bibitem{HB-TI-YM:14}
H.~Bevrani, T.~Ise, and Y.~Miura, ``Virtual synchronous generators: A survey
  and new perspectives,'' \emph{International Journal of Electrical Power \&
  Energy Systems}, vol.~54, 2014.

\bibitem{SD-JA:13}
S.~D'Arco and J.~Suul, ``Virtual synchronous machines -- classification of
  implementations and analysis of equivalence to droop controllers for
  microgrids,'' in \emph{IEEE POWERTECH}, 2013.

\bibitem{MJ-SWHDH-WLK-FJA:06}
J.~Morren, S.~W. De~Haan, W.~L. Kling, and J.~Ferreira, ``Wind turbines
  emulating inertia and supporting primary frequency control,'' \emph{IEEE
  Transactions on power systems}, vol.~21, no.~1, pp. 433--434, 2006.

\bibitem{MK-TB-AU-GA:15}
M.~Koller, T.~Borsche, A.~Ulbig, and G.~Andersson, ``Review of grid
  applications with the zurich 1 \{MW\} battery energy storage system,''
  \emph{Electric Power Systems Research}, vol. 120, 2015.

\bibitem{PMA-CSS-GAT-AMC-MEB:15}
P.~Ashton, C.~Saunders, G.~Taylor, A.~Carter, and M.~Bradley, ``Inertia
  estimation of the gb power system using synchrophasor measurements,''
  \emph{IEEE Transactions on Power Systems}, vol.~30, no.~2, 2015.

\bibitem{EE-GV-AT-BK-MM-MO:14}
E.~Ela, V.~Gevorgian, A.~Tuohy, B.~Kirby, M.~Milligan, and M.~O'Malley,
  ``Market designs for the primary frequency response ancillary service -- part
  i: Motivation and design,'' \emph{IEEE Transactions on Power Systems},
  vol.~29, no.~1, 2014.

\bibitem{TSB-TL-DJH:15}
T.~S. Borsche, T.~Liu, and D.~J. Hill, ``Effects of rotational inertia on power
  system damping and frequency transients,'' in \emph{54th IEEE Conference on
  Decision and Control}, 2015.

\bibitem{KZ-JCD-KG:96}
K.~Zhou, J.~C. Doyle, and K.~Glover, \emph{Robust and Optimal Control}, 1996.

\bibitem{EL-SZ:12}
E.~Lovisari and S.~Zampieri, ``Performance metrics in the average consensus
  problem: a tutorial,'' \emph{Annual Reviews in Control}, vol.~36, no.~1,
  2012.

\bibitem{BB-MRJ-PM-SP:12}
B.~Bamieh, M.~R. Jovanovic, P.~Mitra, and S.~Patterson, ``Coherence in
  large-scale networks: Dimension-dependent limitations of local feedback,''
  vol.~57, no.~9, pp. 2235--2249, 2012.

\bibitem{MF-FL-MRJ:14}
M.~Fardad, F.~Lin, and M.~R. Jovanovic, ``Design of optimal sparse
  interconnection graphs for synchronization of oscillator networks,''
  \emph{IEEE Transactions on Automatic Control}, vol.~59, no.~9, 2014.

\bibitem{MF-XZ-FL-MRJ:14}
M.~Fardad, X.~Zhang, F.~Lin, and M.~R. Jovanovi\'c, ``On the properties of
  optimal weak links in consensus networks,'' in \emph{53rd IEEE Conference on
  Decision and Control}, 2014.

\bibitem{TS-IS-JL-FD:15}
\BIBentryALTinterwordspacing
T.~Summers, I.~Shames, J.~Lygeros, and F.~D{\"o}rfler, ``Topology design for
  optimal network coherence,'' in \emph{European Control Conference}, 2015,
  {Available at \url{http://arxiv.org/abs/1411.4884}}. [Online]. Available:
  \url{http://arxiv.org/abs/1411.4884}
\BIBentrySTDinterwordspacing

\bibitem{MS-NM:14}
M.~Siami and N.~Motee, ``Systemic measures for performance and robustness of
  large-scale interconnected dynamical networks,'' in \emph{53rd IEEE
  Conference on Decision and Control}, 2014.

\bibitem{ES-BB-DFG:15}
E.~Sj{\"o}din, B.~Bamieh, and D.~F. Gayme, ``The price of synchrony: Evaluating
  the resistive losses in synchronizing power networks,'' \emph{IEEE
  Transactions on Control of Network Systems}, vol.~2, no.~3, 2015.

\bibitem{FD-MJ-MC-FB:13a}
F.~D{\"o}rfler, M.~R. Jovanovic, M.~Chertkov, and F.~Bullo,
  ``Sparsity-promoting optimal wide-area control of power networks,''
  \emph{IEEE Transactions on Power Systems}, vol.~29, no.~5, pp. 2281--2291,
  September 2014.

\bibitem{XW-FD-MJ:15a}
X.~Wu, F.~D{\"o}rfler, and M.~R. Jovanovic, ``Input-output analysis and
  decentralized optimal control of inter-area oscillations in power systems,''
  \emph{IEEE Trans. on Power Systems}, vol.~31, no.~3, pp. 2434--2444, May
  2016.

\bibitem{FD-JWSP-FB:14a}
F.~D{\"o}rfler, J.~W. Simpson-Porco, and F.~Bullo, ``Breaking the hierarchy:
  Distributed control \& economic optimality in microgrids,'' \emph{IEEE Trans.
  on Control of Network Systems}, 2015, to appear.

\bibitem{PWS-MAP:98}
P.~W. Sauer and M.~A. Pai, \emph{Power System Dynamics and Stability}, 1998.

\bibitem{PK:94}
P.~Kundur, \emph{Power System Stability and Control}, 1994.

\bibitem{FD-FB:11d}
F.~D{\"o}rfler and F.~Bullo, ``{K}ron reduction of graphs with applications to
  electrical networks,'' \emph{IEEE Transactions on Circuits and Systems~I:
  Regular Papers}, vol.~60, no.~1, 2013.

\bibitem{QCZ-TH:13}
Q.-C. Zhong and T.~Hornik, \emph{Control of Power Inverters in Renewable Energy
  and Smart Grid Integration}.\hskip 1em plus 0.5em minus 0.4em\relax
  Wiley-IEEE Press, 2013.

\bibitem{JS-DZ-RO-AS-TS-JR:15}
J.~Schiffer, D.~Zonetti, R.~Ortega, A.~Stankovic, T.~Sezi, and J.~Raisch, ``A
  survey on modeling of microgrids - from fundamental physics to phasors and
  voltage sources,'' \emph{Automatica}, vol.~74, 2016.

\bibitem{JS-DG-JR-TS:13}
J.~Schiffer, D.~Goldin, J.~Raisch, and T.~Sezi, ``Synchronization of
  droop-controlled autonomous microgrids with distributed rotational and
  electronic generation,'' Florence, Italy, Dec. 2013, pp. 2334--2339.

\bibitem{IAH-EMF:08}
I.~A. Hiskens and E.~M. Fleming, ``Control of inverter-connected sources in
  autonomous microgrids,'' in \emph{American Control Conference, 2008}.\hskip
  1em plus 0.5em minus 0.4em\relax IEEE, 2008, pp. 586--590.

\bibitem{EJC-MBW-SPB:08}
E.~J. Cand\`es, M.~B. Wakin, and S.~P. Boyd, ``Enhancing sparsity by reweighted
  $\ell_1$ minimization,'' \emph{Journal of Fourier Analysis and Applications},
  vol.~14, no.~5, pp. 877--905, 2008.

\end{thebibliography}

\end{document}